\documentclass{amsart}

\usepackage{mathtools,amssymb}
\usepackage[pdftex]{hyperref}
\usepackage{xcolor}
\usepackage{graphicx}
\usepackage{ytableau}
\usepackage{tikz}
\usetikzlibrary{calc, shapes, backgrounds,arrows,positioning,plotmarks}

\theoremstyle{plain}
\newtheorem{theorem}{Theorem}[section]
\newtheorem{lemma}[theorem]{Lemma}

\newtheorem{prob}[theorem]{Problem}
\newtheorem{proposition}[theorem]{Proposition}

\newtheorem{corollary}[theorem]{Corollary}

\theoremstyle{definition}
\newtheorem{definition}[theorem]{Definition}
\newtheorem{example}[theorem]{Example}
\newtheorem{remark}[theorem]{Remark}

\numberwithin{equation}{section}

\newcommand{\bbb}{\mathsf{b}}
\newcommand{\sslash}{\mathbin{/\mkern-6mu/}}
\newcommand{\Gr}{\mathrm{Gr}}
\newcommand{\Flags}{\mathit{F\ell}}

\newcommand{\inc}{\ensuremath{\mathrm{Inc}}}

\newcommand{\nc}{\mathcal{W}}
\newcommand{\rstab}{\mathcal{J}}
\newcommand{\rstabgen}{\overline{\mathcal{J}}}
\newcommand{\rspoly}{\mathbf{J}}
\newcommand{\inv}{\mathrm{inv}}
\newcommand{\sgn}{\mathrm{sgn}}
\newcommand{\setpart}{\Pi}

\newcommand{\SL}{\mathrm{SL}}
\newcommand{\fsl}{\mathfrak{sl}}

\definecolor{darkblue}{rgb}{0.0,0,0.7}

\newcommand{\newword}[1]{\textcolor{darkblue}{\textbf{\emph{#1}}}}

\title[Web invariants for noncrossing partitions]{A web basis of invariant polynomials \linebreak from noncrossing partitions}
\author{Rebecca Patrias}
\address[RP]{\parbox{\linewidth}{Department of Mathematics, University of St.\ Thomas, St.\ Paul, MN, 55105, USA}}
\email{\parbox[t]{\linewidth}{rebecca.patrias@stthomas.edu}}
\author{Oliver Pechenik}
\address[OP]{\parbox{\linewidth}{Department of Combinatorics \& Optimization, University of Waterloo, Waterloo, ON, N2L 3G1, Canada}}
\email{\parbox[t]{\linewidth}{oliver.pechenik@uwaterloo.ca}}
\author{Jessica Striker}
\address[JS]{\parbox{\linewidth}{Department of Mathematics, North Dakota State University, Fargo, ND, 58102, USA}}
\email{\parbox[t]{\linewidth}{jessica.striker@ndsu.edu}}

\begin{document}
\maketitle

\begin{abstract}
The irreducible representations of symmetric groups can be realized as certain graded pieces of invariant rings, equivalently as global sections of line bundles on partial flag varieties. There are various ways to choose useful bases of such Specht modules $S^\lambda$. Particularly powerful are web bases, which make important connections with cluster algebras and quantum link invariants. Unfortunately, web bases are only known in very special cases---essentially, only the cases $\lambda=(d,d)$ and $\lambda=(d,d,d)$. Building on work of B.~Rhoades (2017), we construct an apparent web basis of invariant polynomials for the $2$-parameter family of Specht modules with $\lambda$ of the form $(d,d,1^\ell)$. The planar diagrams that appear are noncrossing set partitions, and we thereby obtain geometric interpretations of earlier enumerative results in combinatorial dynamics.
\end{abstract}

\section{Introduction}

\emph{Specht modules} $S^\lambda$, indexed by integer partitions, are the irreducible complex representations of symmetric groups $S_n$. Unsurprisingly, there are many different ways to construct and describe Specht modules. These various constructions yield distinct linear bases of $S^\lambda$ that are variously well adapted to one task or another. Understanding the resulting basis changes often leads to  deep and hard problems \cite{Kazhdan.Lusztig:1,Kazhdan.Lusztig:2,Elias.Williamson,Khovanov.Kuperberg,Russell.Tymoczko:sl2,Rhoades:polytabloid,Russell.Tymoczko:sl3,Im.Zhu,Hwang.Jang.Oh}.

An important construction of the Specht module is as a space of invariant polynomials for a Lie group of block lower-triangular matrices, or equivalently as global sections of a line bundle on a \emph{partial flag variety}. In this avatar, \emph{standard monomial theory} endows each $S^\lambda$ with a natural basis of polynomials, encoded by \emph{standard Young tableaux} of shape $\lambda$. For textbook treatment of this classical material, see, e.g., \cite[Chapters 8 \& 9]{Fulton:YoungTableaux}, \cite[Chapter 14]{Miller.Sturmfels}, \cite[Chapter 10]{Lakshmibai.Raghavan}, or \cite[Chapters 4, 6, \& 12]{Lakshmibai.Brown}.

In special cases, this realization of $S^\lambda$ has another remarkable basis, consisting of a different set of invariant polynomials, encoded by planar diagrams called \emph{webs}. Webs were introduced by G.~Kuperberg \cite{Kuperberg}, building on the diagrammatics of the Temperley--Lieb algebra \cite{Temperley.Lieb}. The web bases have applications to representations of \emph{quantum groups} (e.g., \cite{Kuperberg,Khovanov.Kuperberg}, \emph{cluster algebras} (e.g., \cite{Fomin.Pylyavskyy,Fraser}), the geometry of \emph{Springer fibres} (e.g., \cite{Mazorchuk.Stroppel,Stroppel.Webster}), and \emph{quantum link invariants} (e.g., \cite{Jones,Khovanov}). As a combinatorial application, web bases were used in \cite{Petersen.Pylyavskyy.Rhoades} to elucidate the orbit structure of certain standard Young tableaux under \emph{promotion}, a instance of the \emph{cyclic sieving phenomenon} \cite{Reiner.Stanton.White} previously established through more difficult tools in \cite{Rhoades:thesis}.

While definitions have been proposed for $\fsl_r$ webs with $r\geq 3$ \cite{Kim,Morrison,Westbury,Fontaine}, the resulting constructions are not entirely satisfactory from our perspective, as they lack important combinatorial properties. In particular, web bases have only been successfully applied to tableau combinatorics in the $2$- and $3$-row rectangular cases $\lambda = (d,d)$ and $\lambda=(d,d,d)$.

The goal of this paper is to develop a web basis for the two-parameter family of Specht modules $S^\lambda$ with $\lambda$ of the form $(d,d,1^\ell)$. (Here, $1^\ell$ is shorthand for $\ell$ parts each of size $1$.) We call these partitions \newword{pennants} after the visual appearance of their Young diagrams (cf.~Figure~\ref{fig:pennant}). We were guided to the study of pennant Specht modules by considerations in the combinatorics of tableau dynamics; however, this family of partitions has attracted combinatorial interest since work of R.~Stanley \cite{Stanley:pennant} in the mid-1990s, based on unpublished observations of K.~O'Hara and A.~Zelevinsky (see also \cite[$\mathsection 9$]{Armstrong.Reiner.Rhoades} for relations to cluster theory). We find that pennant Specht modules have a useful web basis, directly extending the Temperley--Lieb basis for the case $\ell = 0$.

\begin{figure}[ht]
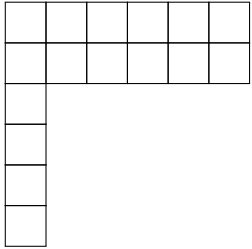

\ydiagram{6,6,1,1,1,1}
\caption{The Young diagram of the pennant partition $\lambda = (6,6,1^4)$.}\label{fig:pennant}
\end{figure}

Standard Young tableaux of this shape $(d,d,1^\ell)$ are in bijection with \emph{increasing tableaux} of shape $(d+\ell,d+\ell)$ with entries at most $2d+\ell$~\cite{Pechenik}.
Increasing tableaux are analogues of standard Young tableaux, useful in the study of $K$-theoretic Schubert calculus. In \cite{Thomas.Yong:K}, H.~Thomas and A.~Yong introduced a \emph{$K$-jeu de taquin} through which increasing tableaux calculate the $K$-theory structure coefficients of \emph{Grassmannians} and other \emph{minuscule varieties}  \cite{Clifford.Thomas.Yong,Buch.Samuel,Pechenik.Yong:KT}. The combinatorics of $K$-jeu de taquin has since found combinatorial applications in the study of \emph{plane partitions} \cite{Dilks.Pechenik.Striker,Vorland,Hamaker.Patrias.Pechenik.Williams,Patrias.Pechenik} and longest strictly increasing subsequences of words \cite{Thomas.Yong:Plancherel}.

Analogous to the cyclic sieving theorems of \cite{Rhoades:thesis,Petersen.Pylyavskyy.Rhoades} for standard Young tableaux, O.~Pechenik \cite{Pechenik} gave a cyclic sieving to describe the orbit structure of $2$-row increasing tableaux under a $K$-jeu de taquin analogue of promotion. The $\fsl_2$ webs relevant to promotion of $2$-row \emph{standard Young tableaux} may be combinatorially identified with \emph{noncrossing matchings}, special cases of \emph{noncrossing set partitions} (see \cite{Stanley:Catalan} for enumerative exploration of these combinatorial objects). The cyclic sieving of \cite{Pechenik} relied on the combinatorics of noncrossing set partitions without singleton blocks, leading Pechenik to write ``it is tempting to think of noncrossing partitions without singletons as `$K$-webs' for $\fsl_2$, although their representation-theoretic significance is unknown.'' However, the proof in \cite{Pechenik} was by explicit calculation, rather than representation-theoretic arguments.

In \cite{Rhoades:skein}, B.~Rhoades provided representation-theoretic meaning to noncrossing partitions without singletons by combinatorially reconstructing the Specht module $S^{(d,d,1^\ell)}$ with these combinatorial diagrams as an apparently new ``skein'' basis.  Thereby, Rhoades established the first algebraic proofs and interpretations of some theorems from \cite{Reiner.Stanton.White} and \cite{Pechenik}. Nonetheless, the results remained somewhat mysterious, since Rhoades' definitions were highly non-obvious and the verifications that they gave a module structure involved many pages of laborious calculations. Moreover, \cite[$\mathsection$7]{Rhoades:skein} noted an incompatibility of signs between different constructions in the paper; specifically, the signs appearing in the $S_n$ action on general set partitions are slightly different from those used for \emph{almost noncrossing set partitions}.

Our main result is to realize Rhoades' skein basis (up to some changes in sign) in a geometrically-natural setting. More precisely, we associate to each noncrossing set partition without singletons a global section $[\pi]$ of a line bundle on the $2$-step partial flag variety $\Flags(2,\ell+2;n)$, yielding a basis of the pennant Specht module $S^{(d,d,1^\ell)}$ that is equivalent to Rhoades' skein basis up to signs. The polynomial $[\pi]$ will be defined in Section~\ref{sec:main} in terms of \emph{jellyfish tableaux}, which we introduce (building on ideas we learned from \cite{Reiner.Shimozono}).

\begin{theorem}\label{thm:main}
The set $\{ [\pi] : \pi \in \nc(n,d) \}$ 
of invariants of noncrossing partitions of $n$ with $d$ blocks and no singletons forms a basis for 
$S^{\lambda}$ where $\lambda = (d,d,1^{n-2d-2})$. 
Moreover, up to predictable signs, the long cycle $c = n12 \dots (n-1) \in S_n$ acts by rotation of diagrams and the long element $w_0 = n(n-1) \dots 21 \in S_n$ acts by reflection. Specifically, we have
\[
w_0 \cdot [\pi] = (-1)^{\binom{n}{2}} [\mathtt{refl}(\pi)] \text{ and } c \cdot [\pi] = (-1)^{n-1} [\mathtt{rot}(\pi)].
\]
\end{theorem}

Our construction directly extends that of the $\fsl_2$ web basis for the Specht module $S^{(d,d)}$. Our construction partially realizes the dream from \cite{Pechenik} of interpreting noncrossing partitions without singletons as ``$K$-webs'' for $\fsl_2$. However, they turn out to be webs, not for $\fsl_2$, but rather for a block lower-triangular Lie subalgebra of $\fsl_{\ell+2}$, and any connection to $K$-theoretic geometry is not yet apparent. An asset of our construction is that we actually obtain an invariant polynomial $[\pi] \in S^\lambda$ for each set partition $\pi$ (not necessarily noncrossing) with the actions of the permutations $c$ and $w_0$ still given as in Theorem~\ref{thm:main}. With this construction, it moreover becomes straightforward to obtain ``uncrossing rules,'' describing the expansion of $[\pi]$ for crossing $\pi$ in the noncrossing basis.

\begin{remark} 
A different algebraic realization of the skein basis in a space of \emph{fermionic diagonal harmonics} appeared recently in work of Rhoades with J.~Kim \cite{Kim.Rhoades}. While our work was carried out independently, our construction as a space of invariant polynomials appears to be a concretization of their more abstract theory. In particular, their \emph{block operators} appear to correspond to our determinants and they satisfy a $5$-term relation analogous to our Theorem~\ref{thm:5term}. It would likely be valuable to work out the details of this correspondence.
\end{remark}

To our knowledge, Theorem~\ref{thm:main} is the first development of web bases for any family of Lie algebras of arbitrarily-high dimension. We have yet to fully explore the consequences of this idea. Our first applications are to understanding enumerative questions in combinatorial dynamics. However, it seems reasonable to expect that there is a quantum group whose representation theory is also governed by this diagrammatic basis. We do not expect this theory to yield a new quantum link invariant; however, it could perhaps be applied to obtain quantum invariants for spatial embeddings of hypergraphs. Similarly, we hope that our constructions may yield insights into the topology of the Springer fiber for shape $(d,d,1^\ell)$. The dynamics of $K$-promotion on arbitrary rectangular increasing tableaux are mysterious but known to be very complicated \cite{Pechenik,Patrias.Pechenik}. However, the $3$-row rectangular case is believed to be tractable \cite[Conjecture~4.12]{Dilks.Pechenik.Striker}. Preliminary investigations with Julianna Tymoczko suggest that this case is also amenable to study via a web basis of planar hypergraphs.

This paper is organized as follows. In Section~\ref{sec:background}, we recall necessary background material on Specht modules, partial flag varieties, and tableau combinatorics. In particular,  Subsection~\ref{sec:grassmannian} outlines the Grassmannian case and the construction of bases of the Specht module $S^{(d,d)}$ in terms of tableaux and noncrossing matchings, which our main result specializes to in the case $\ell=0$. Subsection~\ref{sec:flagvarieties} extends this  to partial flag varieties and constructs a polynomial basis indexed by tableaux of the  Specht module $S^{\lambda}$; this includes our case of interest, $\lambda = (d,d,1^{\ell})$.  In Section~\ref{sec:main}, we introduce jellyfish tableaux and invariant polynomials for set partitions and prove our main theorems. Section~\ref{sec:final} contains some final remarks on cyclic sieving, uncrossing rules, and open problems.

\section{Background}\label{sec:background}

\subsection{Partitions and Young diagrams}
A \newword{partition} $\lambda = (\lambda_1 \geq \lambda_2 \geq \dots \geq \lambda_r)$ is a finite weakly decreasing sequence of positive integers. The number of terms in this sequence is the \newword{length} $r$ of the partition. We will use shorthands for partitions of the form $(4^2,1^3) = (4,4,1,1,1)$. (This notation can be justified by thinking of the exponents as denoting products in the free monoid on the alphabet of positive integers, rather than ordinary products of numbers.)

We commonly draw a partition $\lambda$ by placing $\lambda_i$ left-justified boxes in row $i$, where we number our rows from top to bottom. (In other words, we will be using ``English'' conventions and matrix coordinates.) This box diagram is called the \newword{Young diagram} of $\lambda$; however, we generally conflate $\lambda$ with its Young diagram. For example, if $\lambda = (4,4,1^3)$, its Young diagram is $\ytableausetup{smalltableaux} \ydiagram{4,4,1,1,1}$. \ytableausetup{nosmalltableaux}

\subsection{Grassmannians and webs}
\label{sec:grassmannian}
The complex \newword{Grassmannian} $\Gr_k(n)$ is the parameter space for $k$-dimensional linear subspaces of $\mathbb{C}^n$. Our focus in this paper is on generalizing constructions for $\Gr_2(n)$, so we begin by reviewing these. In the next subsection, we will describe more general constructions for partial flag varieties, but we think it is valuable to first fix ideas and notation in the Grassmannian case.  Textbook references for this subsection and the next include \cite[Chapters 8 \& 9]{Fulton:YoungTableaux}, \cite[Chapter 14]{Miller.Sturmfels}, \cite[Chapter 10]{Lakshmibai.Raghavan}, and \cite[Chapters 4, 6, \& 12]{Lakshmibai.Brown}.

Let $M_n$ denote the matrix $\begin{pmatrix}
x_{1} & x_{2} & \dots & x_{n}\\
y_{1} & y_{2} & \dots & y_{n}
\end{pmatrix}$ of $2n$ distinct indeterminates and let $\SL_2(\mathbb{C})$ act by left multiplication. This gives an action of $\SL_2(\mathbb{C})$ on the ring $\mathbb{C}[M_n]$ of polynomials in these $2n$ variables, which we think of as the coordinate ring of a $2n$-dimensional affine space $\mathbb{C}^{2n}$. A classical task in invariant theory is to characterize those polynomials in the invariant subspace $\mathbb{C}[M_n]^{\SL_2}$ of this action. Classically, the answer is that $\mathbb{C}[M_n]^{\SL_2}$ is generated as an algebra by the $2 \times 2$ minors of $M_n$. 

By definition, $\mathbb{C}[M_n]^{\SL_2}$ is the coordinate ring of the GIT (geometry invariant theory) quotient $\mathbb{C}^{2n} \sslash \SL_2$, which is the affine cone over $\Gr_2(n)$. A $2 \times 2$ minor of $M_n$ is specified by a pair of column indices $1 \leq i < j \leq n$. Commonly, we identify the $(i<j)$--minor  
$\begin{vmatrix}x_i & x_j\\ y_i & y_j\end{vmatrix}$
 of $M_n$ with the \newword{Pl\"ucker variable} $p_{ij}$ on the complex projective space $\mathbb{P}^{\binom{n}{2}-1}$. Realizing a $2$-dimensional linear subspace of $\mathbb{C}^n$ as a rank $2$ complex matrix of shape $2 \times n$, we may evaluate all of its $2 \times 2$ minors and map it to the point in $\mathbb{P}^{\binom{n}{2}-1}$ with those Pl\"ucker coordinates. This construction is well-defined, since different matrix representations of the same $2$-plane will yield the same set of $2 \times 2$ minors, up to global scaling by a nonzero constant. We thereby realize the Grassmannian $\Gr_2(n)$ as a smooth projective subvariety of $\mathbb{P}^{\binom{n}{2}-1}$, embedded via this \newword{Pl\"ucker embedding}. The Grassmannian, in its Pl\"ucker embeddding, is cut out scheme-theoretically by quadratic relations in the Pl\"ucker variables, called the \newword{Pl\"ucker relations}. Hence, the homogeneous coordinate ring of $\Gr_2(n)$ is a polynomial ring in the Pl\"ucker variables modulo these Pl\"ucker relations.

We then have an isomorphism between $\mathbb{C}[M_n]^{\SL_2}$ and the homogeneous coordinate ring of $\Gr_2(n)$, except that the gradings differ by a factor of $2$. The degree $2d$ homogeneous part of $\mathbb{C}[M_n]^{\SL_2}$ corresponds to degree $d$ polynomials in the Pl\"ucker variables, since $2 \times 2$ minors are quadratic. In particular, $\mathbb{C}[M_n]^{\SL_2}$ is supported only in even degrees. We will move back and forth between these gradings as convenient.

The polynomial ring $R=\mathbb{C}[p_{ij}: 1 \leq i < j \leq n]$ is the homogeneous coordinate ring of $\mathbb{P}^{\binom{n}{2}-1}$. Such polynomials of homogeneous degree $d$ form the global sections of the line bundle $\mathcal{O}(d)$ on $\mathbb{P}^{\binom{n}{2}-1}$. Similarly, if we quotient $R$ by the Pl\"ucker relations, the resulting functions of homogeneous degree $d$ are the global sections of the pullback line bundle $\mathcal{O}_{\Gr_2(n)}(d)$. Of course, we may then identify these functions with the degree $2d$ part of the invariant ring $\mathbb{C}[M_n]^{\SL_2}$.

For each $d$, the degree $2d$ homogeneous part of $\mathbb{C}[M_n]^{\SL_2}$ is clearly finite-dimensional, since it is spanned by $d$-fold products of Pl\"ucker variables and there are finitely many such variables. There are two convenient graphical ways to encode such a $d$-fold product.

In the tableau encoding, we fill a $2 \times d$ grid $T$ with positive integers weakly between $1$ and $n$. For each Pl\"ucker variable $p_{ij}$, we fill a column of $T$ with $i$ and $j$. Note that this grid is exactly the Young diagram of the partition $(d,d)$. We can choose to insist that, whenever $i<j$, we write $i$ in the top row and $j$ in the bottom row. The order of the columns is also inconsequential, since a product of Pl\"ucker variables is independent of the order of the factors. Hence, we may assume that we have sorted the columns of $T$ so that the entries of the first row weakly increase from left to right. Thus, a $d$-fold product of Pl\"ucker variables is encoded by an array such as in Figure~\ref{fig:bigproduct} (where in this example we are taking $d=7$ and $n\geq 7$). 
\begin{figure}[htbp]
\begin{center}
$p_{12}p_{14}p_{17}p_{25}p_{34}^2p_{47}$

\medskip

$\begin{vmatrix}
x_1 & x_2 \\
y_1 & y_2  
\end{vmatrix} 
\begin{vmatrix}
x_1 & x_4 \\
y_1 & y_4  
\end{vmatrix} 
\begin{vmatrix}
x_1 & x_7 \\
y_1 & y_7  
\end{vmatrix} 
\begin{vmatrix}
x_2 & x_5 \\
y_2 & y_5  
\end{vmatrix}
\begin{vmatrix}
x_3 & x_4 \\
y_3 & y_4  
\end{vmatrix}^2
\begin{vmatrix}
x_4 & x_7 \\
y_4 & y_7  
\end{vmatrix}$

\medskip

\ytableaushort{1112334,2475447}
\end{center}
\caption{An example of a 7-fold product of Pl\"ucker variables and its tableau encoding}
\label{fig:bigproduct}
\end{figure}
Note that the bottom row of the array satisfies no particular increasingness condition. We say that such a filling of a Young diagram is a \newword{semistandard tableau} if the entries of each row weakly increase from left to right, just as we insisted for the entries of the top row.

A fundamental result of \emph{standard monomial theory} is that, for each $d$,
 a linear basis for the degree $2d$ homogeneous part of $\mathbb{C}[M_n]^{\SL_2}$ is given by those products of Pl\"ucker variables corresponding to semistandard tableaux of shape $(d,d)$. This fact allows one, for example, to combinatorially determine the dimensions of these spaces.
 
In the web encoding, we cyclically place $n$ labeled points around the boundary of a disk, and then draw $d$ arcs through the interior of this disk joining boundary points. Precisely, for each Pl\"ucker variable $p_{ij}$,  we draw an arc from vertex $i$ to vertex $j$. Note that some arcs may necessarily cross other arcs, and indeed some pairs of arcs may share the same endpoints. Letting $d=7$ and $n=8$, the $d$-fold product encoded by our example array above may be alternatively encoded by the diagram in Figure~\ref{fig:web_for_bigproduct}.

\begin{figure}[htbp]
\begin{center}
\includegraphics[width=1.5in]{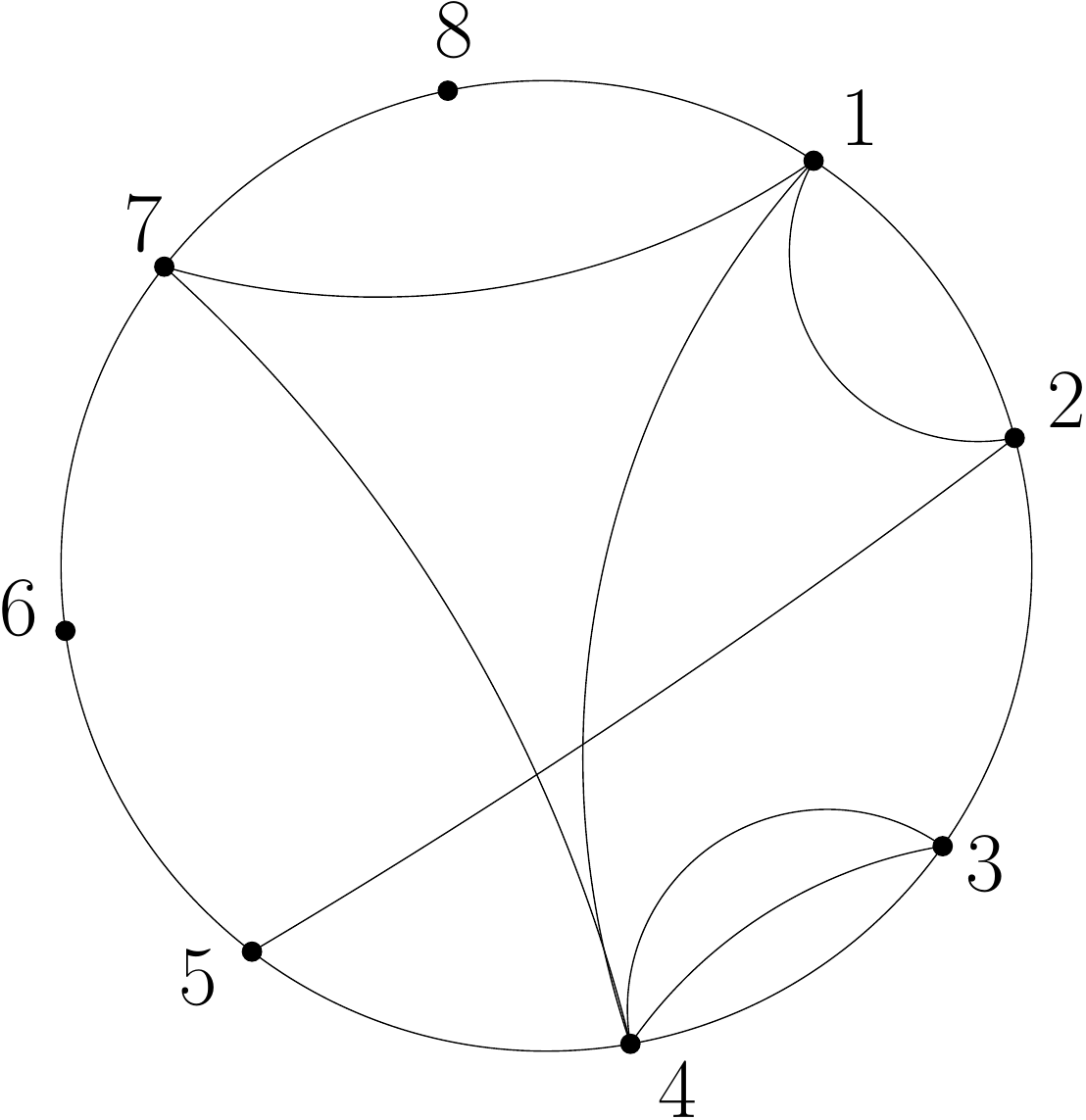}
\end{center}
\caption{A diagramatic encoding of the 7-fold product of Pl\"ucker variables from Figure~\ref{fig:bigproduct}}
\label{fig:web_for_bigproduct}
\end{figure}

Let $T_n$ denote the rank $n$ torus of diagonal invertible $n \times n$ matrices. Then $T_n$ acts on $M_n$ by right multiplication, scaling the columns. We thereby also have an action of $T_n$ on the invariant ring $\mathbb{C}[M_n]^{\SL_2}$. The $T_n$ action breaks the degree $2d$ homogeneous part of $\mathbb{C}[M_n]^{\SL_2}$ into a direct sum of \emph{weight spaces}. Specifically, we obtain an extra $\mathbb{Z}^n$-grading on invariant polynomials. Let $e_i$ denote the $i$th standard basis vector of $\mathbb{Z}^n$. Then we treat each Pl\"ucker variable $p_{ij}$ as having $\mathbb{Z}^n$-degree $e_i + e_j$. For each $v \in \mathbb{Z}^n$ with coordinates summing to $2d$, we may then consider the part of $\mathbb{C}[M_n]^{\SL_2}$ with $\mathbb{Z}^n$-degree $v$, a subspace of the space of degree $2d$ polynomials. It is straightforward to see that, for $v = (v_1, \dots, v_n)$, the homogeneous degree $v$ part of $\mathbb{C}[M_n]^{\SL_2}$ corresponds to those tableaux in which each value $i$ appears exactly $v_i$ times, or alternatively, to those arc diagrams where boundary vertex $i$ has degree $v_i$.
 
 We are particularly interested in the case $n=2d$ and the degree $(1,1, \dots, 1)$ homogeneous part of $\mathbb{C}[M_n]^{\SL_2}$. In this case, we have a basis given by semistandard tableaux where each number from $1$ to $n=2d$ appears exactly once. Such tableaux are called \newword{standard}. However, the theory of $\fsl_2$ webs gives us another, completely different basis for this space. Clearly, the degree $(1,1, \dots, 1)$ homogeneous part is spanned by products of Pl\"ucker variables corresponding to \newword{matchings}, i.e., arc diagrams where each boundary vertex is paired to exactly one other boundary vertex by an arc. It turns out that a linear basis is given by those matchings that are \newword{noncrossing}, i.e., such that the arcs can all be drawn so as to be pairwise nonintersecting. Noncrossing matchings are also known as \newword{$\fsl_2$ webs}, and the corresponding basis is called the \newword{$\fsl_2$ web basis} or the \newword{Temperley--Lieb basis}.

 The tableau and web bases are genuinely different from each other. For example, if $m=2$ and $n=4$, there are three products of pairs of distinct Pl\"ucker variables to consider, as shown in Figure~\ref{fig:basesaredifferent}. The standard tableau basis is given by the first two of these products, whereas the $\fsl_2$ web basis is given by the first and the third.
 
 \begin{figure}[ht]
 \[
 \begin{array}{ccc}
 	 \quad p_{12}p_{34}    & p_{13}p_{24}   &    p_{14}p_{23} \quad \\ \\    
   \begin{vmatrix}x_1 & x_2\\ y_1 & y_2\end{vmatrix} \begin{vmatrix}x_3 & x_4\\ y_3 & y_4\end{vmatrix}    & \begin{vmatrix}x_1 & x_3\\ y_1 & y_3\end{vmatrix} \begin{vmatrix}x_2 & x_4\\ y_2 & y_4\end{vmatrix}   &   \begin{vmatrix}x_1 & x_4\\ y_1 & y_4\end{vmatrix} \begin{vmatrix}x_2 & x_3\\ y_2 & y_3\end{vmatrix} \\ \\
 	\quad \quad \ytableaushort{13,24}     & \ytableaushort{12,34} &    \ytableaushort{12,43} \quad \quad
 \end{array}
\]
\begin{center}
 \includegraphics[width=2.75in]{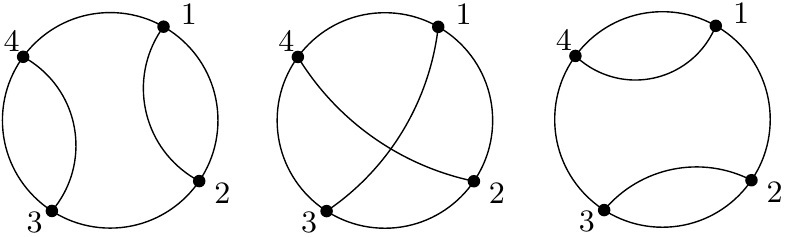}
\end{center}
\caption{A comparison of the tableau and $\fsl_2$ web bases for the degree $(1,1,1,1)$ part of $\mathbb{C}[M_4]^{\SL_2}$. The tableau basis is given by the first two of these products, corresponding to standard tableaux, whereas the $\fsl_2$ web basis is given by the first and the third, corresponding to nocrossing matchings.}
\label{fig:basesaredifferent}
\end{figure}

Famously, standard Young tableaux of shape $(d,d)$ and noncrossing matchings of $2d$ are both counted by the \emph{Catalan numbers} (see, e.g., \cite{Stanley:Catalan}). There is a well-known bijection between these sets. Although it does not preserve bases (perhaps even \emph{because} it does not preserve bases), this bijection will be important to us later.

The space $\mathbb{C}[M_n]^{\SL_2}$ caries an action of the symmetric group $S_n$ by right multiplication, permuting the columns of $M_n$. For $\mu$ a partition of $2d$, we may consider the union of the homogeneous degree $\alpha$ parts of $\mathbb{C}[M_n]^{\SL_2}$ over all $\alpha$ that are permutations of $\mu$. The $S_n$ action restricts to this union, making it an $S_n$-module; however, the isomorphism type of this representation is not easy to determine in general. Nonetheless, in the case $n=2d$, restricting to the homogeneous degree $(1,1, \dots, 1)$ part yields an irreducible $S_n$-representation, called the \newword{Specht module} $S^{(d,d)}$.

 If instead we let $M_n$ be a $3 \times n$ matrix of distinct indeterminants and consider $\mathbb{C}[M_n]^{\SL_3}$, the story is similar. 
 This invariant ring is generated by the $3 \times 3$ minors, which we identify with the Pl\"ucker coordinates $p_{ijk}$ on $\Gr_3(n)$.
 The homogeneous degree $3d$ part of $\mathbb{C}[M_n]^{\SL_3}$ is the space of global sections of the line bundle $\mathcal{O}_{\Gr_3(n)}(d)$ and has a basis given by semistandard tableaux of shape $(d,d,d)$.

 Again, we have an action of $T_n$ and may consider the weight spaces that this action induces. For $n = 3d$, the homogeneous degree $(1,1, \dots, 1)$ part has a basis given by standard tableaux of shape $(d,d,d)$. The symmetric group $S_n$ acts on this space by right multiplication, making it an $S_n$-module. Indeed, it is an irreducible representation and is a geometric realization of the Specht module $S^{(d,d,d)}$. However, there is again a genuinely different basis of this space encoded by planar diagrams. These diagrams are called \newword{$\fsl_3$ webs} and look, for example, like:
 
 \begin{center}
\includegraphics[width=1.5in]{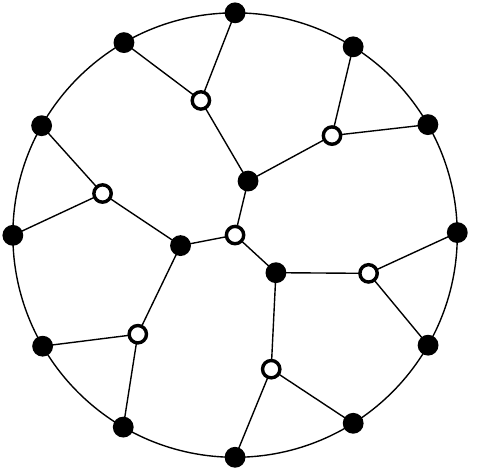}
\end{center}

We omit the rules for how to extract an invariant polynomial from such a diagram, as we will not use them, but see, for example, \cite{BazierMatte.Douville.Garver.Patrias.Thomas.Yildirim} for discussion.

\subsection{Flag varieties}
\label{sec:flagvarieties}
In this subsection, we extend the previous constructions to general partial flag varieties of type $A$, including Grassmannians of $k$-planes with $k>3$. The main difference from the previous constructions (besides some added complexity) is that will not have web bases is these settings. Although we consider arbitrary Specht modules in this subsection, our primary interest will be in those of the form $S^{(d,d,1^\ell)}$; the main result of this paper is to describe an explicit basis of invariant polynomials for those cases, realizing Rhoades' skein basis geometrically.

Let $\lambda$ be any partition and let $\mu$ be the partition whose Young diagram is the transpose of $\lambda$. In particular, if $\lambda = (d,d,1^\ell)$, then $\mu=(\ell+2,2^{d-1})$. Choose $n \geq \mu_1$ and let $M$ denote the $\mu_1 \times n$ matrix 
\[
\begin{pmatrix}
x_{11} & x_{12} & \dots & x_{1n}\\
x_{21} & x_{22} & \dots & x_{2n} \\
\vdots & \vdots & \ddots & \vdots \\
x_{\mu_1 1} & x_{\mu_12} & \dots & x_{\mu_1n}
\end{pmatrix}
\]
of $\mu_1 n$ distinct indeterminants. Now, define another partition $\nu$ as follows: if $\mu_i = \mu_{i+1}$, delete one of them until we obtain a \newword{strict partition} $\nu$ in which all parts are distinct. For example, if $\mu=(\ell+2,2^{d-1})$, then $\nu=(\ell+2,2)$. Consider the set of complex block lower triangular matrices with diagonal block sizes \[
\nu_n, \nu_{n-1} - \nu_n, \nu_{n-2} - \nu_{n-1}, \dots, \nu_1 - \nu_2
\] and each diagonal block of determinant $1$. Note that since $\nu$ is a strict partition, all of these block sizes are positive. (We could have used $\mu$ in place of $\nu$, treating blocks of size $0$ as nonexistent, but this indexing would be less convenient later). These matrices form a subgroup $P$ of the Lie group $\SL_{\mu_1}(\mathbb{C})$, which acts on $\mathbb{C}[M]$. We consider again the invariant subring $\mathbb{C}[M]^P$.

On the one hand, $\mathbb{C}[M]^P$ is the coordinate ring of the GIT quotient $\mathbb{C}^{\mu_1n} \sslash P$, which is the affine cone over the \newword{partial flag variety} $\Flags(\nu_n, \nu_{n-1}, \dots, \nu_1; n)$. This partial flag variety is a parameter space for \newword{partial flags}
\[0 \leq V_n \leq V_{n-1} \leq \dots \leq V_1 \leq \mathbb{C}^n
\]
of nested vector subspaces with $V_i$ of dimension $\nu_i$. In particular, these spaces include the Grassmannians, as $\Gr_k(n) = \Flags(k;n)$. Our main interest is in the $2$-step flag variety $\Flags(2, \ell+2; n)$.

On the other hand, classical invariant theory tells us algebraic generators for $\mathbb{C}[M]^P$. For $k$ any of the parts of $\nu$, let $p_{i_1i_2\ldots i_k}$ denote the $k \times k$ minor of $M$ that uses the top $k$ rows and the columns indexed $i_1, i_2, \dots, i_k$. Then $\mathbb{C}[M]^P$ is algebraically generated by the set 
\[
\{p_{i_1i_2\ldots i_k} : k \in \nu, 1 \leq i_1 < i_2 < \dots < i_k \leq n \}.
\] We call each $p_{i_1i_2\ldots i_k}$ a \newword{$k$-Pl\"ucker variable}. 

Representing a partial flag as a $\mu_1 \times n$ complex matrix such that $V_i$ is the span of the first $\nu_i$ rows, computing all of the Pl\"ucker variables on this matrix embeds $\Flags(\nu_n, \nu_{n-1}, \dots, \nu_1; n)$ in a product of projective spaces 
\[
\mathbb{P}^{\binom{n}{\nu_n}-1} \times \mathbb{P}^{\binom{n}{\nu_{n-1}}-1} \times \dots \times \mathbb{P}^{\binom{n}{\nu_1}-1}
\]
as a smooth subvariety. We also call this map the \newword{Pl\"ucker embedding}. 

Now, consider the invariants in $\mathbb{C}[M]^P$ of the form \[\prod_{i=1}^{\lambda_1} p(\mu_i),\] where each $p(\mu_i)$ is a $\mu_i$-Pl\"ucker variable. For example, with $\lambda=(d,d,1^\ell)$, we are looking at products of one minor of size $\ell+2$ with $d-1$ top-justified minors of size $2$. We consider this set to be the invariants of homogeneous multidegree $\mu$.  (If we write $\mathcal{L}_i$ for the twisting line bundle $\mathcal{O}(1)$ on the projective space $\mathbb{P}^{\binom{n}{\nu_i}-1}$, the invariants are the global sections of the restriction to the partial flag variety of the line bundle
\[
\mathcal{L}_n^{\otimes m_n} \otimes \mathcal{L}_{n-1}^{\otimes m_{n-1}} \otimes \dots \otimes \mathcal{L}_1^{\otimes m_1},
\]
where $m_i$ denotes the multiplicity of $\nu_i$ as a part of $\mu$.)

We may record such a product of Pl\"ucker variables by filling the Young diagram of $\lambda$ such that the $i$th column contains the $\mu_i$ subscripts of the $i$th Pl\"ucker variable.
The key result of standard monomial theory is that this space of invariants has a basis given by such fillings that are semistandard tableaux.

Again, there is an action of the torus $T_n$ by right multiplication on $M$, breaking up our invariants into weight spaces. We now restrict to the case $n = |\lambda|$ and consider the degree $(1,1, \dots, 1)$ homogeneous part. This subspace has a basis given by standard tableaux of shape $\lambda$. This subspace also carries an action of $S_n$ by right multiplication on $M$. It is, in fact, an irreducible module, the \newword{Specht module} $S^\lambda$ (this follows, for example, from the ``Specht polynomial'' description in \cite[\textsection 4.3]{Lakshmibai.Brown} of the Specht module as a quotient of the avatar we have given that has the same dimension). The basis of polynomials in our description is equivalent to the ``polytabloid'' basis discussed in, e.g., \cite[\textsection 7]{Fulton:YoungTableaux} (cf.\ \cite{Howe.Liu.Vaughan}).

\subsection{Increasing tableaux and set partitions}\label{sec:tableaux}

An \newword{increasing tableau} is a semistandard tableau in which rows (like columns) are strictly increasing. Note that standard tableaux are a special case of increasing tableaux. All increasing tableaux in this paper are assumed to be \newword{packed}, meaning that the set of numbers appearing is an initial segment of the set of positive integers. For example, $\ytableausetup{smalltableaux} \ytableaushort{123,3}$ is a packed increasing tableau, whereas $\ytableaushort{125,2} \ytableausetup{nosmalltableaux}$ is an increasing tableau that is not packed. We write $\inc(\lambda)$ for the set of all packed increasing tableaux of shape $\lambda$ and $\inc^q(\lambda)$ for the subset whose maximum entry is $q$.

We are interested in an invertible operator called \newword{$K$-promotion} on increasing tableaux. $K$-promotion was introduced in \cite{Pechenik}, building on the $K$-jeu de taquin theory of \cite{Thomas.Yong:K}, which combinatorially computes structure constants in $K$-theoretic Schubert calculus. (Warning: $K$-promotion is not a restriction to increasing tableaux of M.-P.~Sch\"utzenberger's more classical notion of \emph{promotion} on semistandard tableaux \cite{Schutzenberger}; however, the two operations coincide on standard tableaux.) 

We follow a characterization of $K$-promotion that first appeared in \cite{Dilks.Pechenik.Striker}.
Let $\tau_k$ be a operator that acts on an increasing tableau $T$ by swapping the values $k$ and $k+1$ in all boxes where doing so yields another increasing tableaux. In other words, $\tau_k$ will turn a $k$ in box $\bbb$ into a $k+1$, unless that box shares an edge with a box already labeled $k+1$; similarly, it turns a $k+1$ into a $k$, unless the $k+1$ is already adjacent to a $k$. It is easy to see that $\tau_k$ is an involution. Then $K$-promotion on $T \in \inc^q(\lambda)$ is defined to be the composite operator
\[
\psi^q \coloneqq \tau_{q-1} \circ \dots \circ \tau_2  \circ \tau_1.
\]
Since each $\tau_k$ is an involution, $\psi^q$ is invertible and partitions the set $\inc^q(\lambda)$ into orbits. However, these orbits are generally very complicated and not well understood (see \cite{Pechenik,Bloom.Pechenik.Saracino,Dilks.Pechenik.Striker,Pechenik:frames,Patrias.Pechenik} for partial results and conjectures).

The related operator of \newword{$K$-evacuation} on $\inc^q(\lambda)$ was introduced in \cite{Thomas.Yong:K} and may be defined as
\[
\epsilon^q \coloneqq \psi^1 \circ \dots \circ \psi^{q-1} \circ \psi^q.
\]

There is a bijection $\Psi$ between $\inc^{2d+\ell}(d+\ell,d+\ell)$ and ${\rm SYT}(d,d,1^{\ell})$ given in \cite[Proposition~2.1]{Pechenik}, where ${\rm SYT}(d,d,1^{\ell})$ denotes the set of standard Young tableaux of shape $(d,d,1^\ell)$. Specifically, given $T \in \inc^{2d+\ell}(d+\ell,d+\ell)$, let $A$ be the set of entries that appear twice (once in each row) and let $B$ be the set of entries appearing in the bottom row directly to the right of an element of $A$. Then $\Psi(T)$ is formed by deleting the elements of $A$ from the top row, deleting the elements of $B$ from the second row, and appending the elements of $B$ to the bottom of the first column. See Figure~\ref{fig:bijections} for an example. 

A \newword{set partition} of size $n$ is a collection of pairwise disjoint subsets of $\{1, 2, \dots, n\}$ whose union is $\{1, 2, \dots, n\}$. Each of the subsets is called a \newword{block} of the set partition. We draw a set partition by placing $n$ cyclically labeled points around the boundary of a disk, and then, for each block, drawing the convex hull of the corresponding boundary points on the disk. Observe that a set partition with all blocks of size $2$ is exactly a matching. A \newword{noncrossing partition} is a set partition with the property that the convex hulls (in the graphical depiction above) are pairwise disjoint. In particular, a matching is a noncrossing partition if and only if it is a noncrossing matching, i.e., an $\fsl_2$ web.

A \newword{singleton} in a set partition is a block of size $1$. 
For $n \in \mathbb{Z^+}$, let $\setpart(n)$ be the set of all set partitions of $\{1,2,\ldots,n\}$ with no singletons and let $\setpart(n,d)$ denote the set of all $\pi\in\setpart(n)$ with $d$ blocks. If the blocks of $\pi\in\setpart(n)$ are denoted $\pi_1,\pi_2,\ldots,\pi_d$, we say $\pi=\{\pi_1,\pi_2,\ldots,\pi_d\}$. Let $\nc(n)$ denote the set of all $\pi\in\setpart(n)$ that are noncrossing. Let $\nc(n,d)$ be the set of $\pi\in\nc(n)$ with $d$ blocks.

A useful bijection appeared in \cite[Proposition~2.3]{Pechenik} between $\inc^q(m,m)$ and noncrossing partitions of size $q$ with $q-m$ blocks and no singletons. We briefly recall this map.
Since $T\in \inc^q(m,m)$ is packed, each  $i\in \{1,\ldots,q\}$ appears in $T$ once or twice. If $i$ appears once and is in the first (second) row, $i$ will be the smallest (largest) in its block in the corresponding set partition. 
If $i$ appears twice in $T$, $i$ is neither largest nor smallest in its block in the corresponding set partition. This information is sufficient to construct the entire set partition given the restriction that it be noncrossing. See Figure~\ref{fig:bijections} for an example.

\begin{figure}[ht]
\raisebox{1in}{
\ytableaushort{147,26{10},3,5,8,9} \quad \quad
\ytableaushort{1234678,235689{10}}} \quad \quad
\includegraphics[scale=0.5]{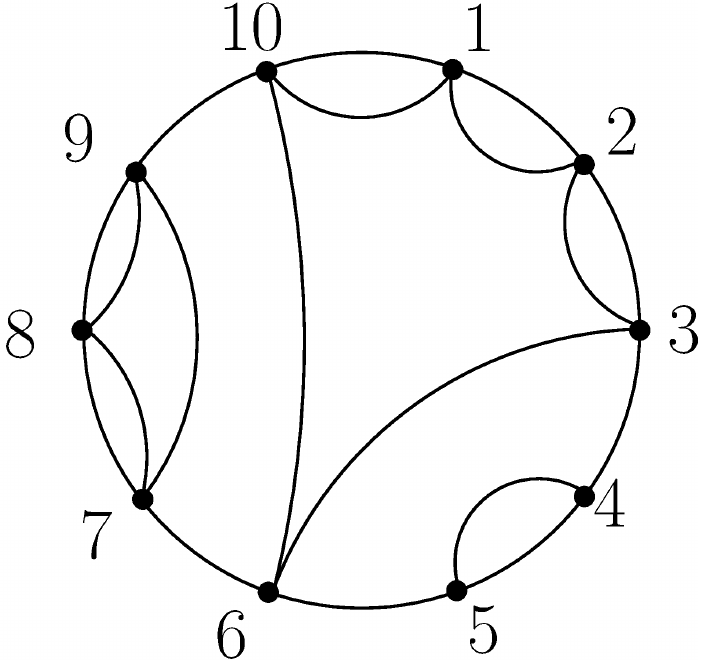} 
\caption{The bijections recalled in Section~\ref{sec:tableaux} map the packed increasing tableau $T \in \inc^{10}(7,7)$ (center) to the standard Young tableau $U \in {\rm SYT}(3,3,1^4)$ (left) and the noncrossing set partition $\pi \in \nc(10,3)$ of $\{1, 2, \dots, 10 \}$ with $3$ blocks and no singletons (right).}\label{fig:bijections}
\end{figure}

This bijection restricts to a bijection between standard tableaux of shape $(m,m)$ and $\fsl_2$ webs with $2m$ boundary vertices. This restriction is classical (see, e.g., \cite{Stanley:Catalan}).

There is also an explicit bijection between standard tableaux of shape $(m,m,m)$ and $\fsl_3$ webs with $3m$ black boundary vertices. This bijection was first given in \cite{Khovanov.Kuperberg} and later with radically simplified description in \cite{Tymoczko}.

These latter bijections all have the property that they carry $K$-promotion of these rectangular tableaux to rotation of the corresponding planar diagram (either a web or a noncrossing partition) \cite{Petersen.Pylyavskyy.Rhoades,Pechenik}. Similarly, they all carry $K$-evacuation to reflection of planar diagrams \cite{Pechenik,Patrias.Pechenik:evacuation}. (Note that, however, the bijection between increasing tableaux and standard Young tableaux is only $K$-evacuation-equivariant \cite[Proposition~5.1]{Pechenik}, and not $K$-promotion equivariant; indeed, the promotion orbits for pennant standard Young tableaux are often much larger than is possible for the corresponding increasing tableaux.) Finally, $K$-promotion of these increasing tableaux is equivariant with \emph{rowmotion} of order ideals in certain graded posets~\cite{Dilks.Pechenik.Striker}; thus, this equivariant bijection to planar diagrams that rotate also provides a good explanation for the order of this much-studied action \cite{Cameron.Fonderflaass, Striker.Williams, Striker.notices}.

\section{Main results}\label{sec:main}

In this section, we prove Theorem~\ref{thm:main}. We begin in Section~\ref{sec:jellyfish} by defining an associated polynomial for each set partition. Section~\ref{sec:lemmas} contains important lemmas for proving our main recurrence, Theorem~\ref{thm:5term}. Section~\ref{sec:5term} states and proves this theorem. Finally, Section~\ref{sec:rest} completes the proof of Theorem~\ref{thm:main}.

\subsection{Web invariants}
\label{sec:jellyfish}
In this subsection, we define a polynomial associated to each set partition; these polynomials will be our web invariants.

Choose and fix a convention for ordering the blocks of a set partition. All of our results will be independent of this convention choice. In particular, the web invariant associated to a set partition does not depend on the choice of ordering.

We begin by defining a set of tableaux that we will use to construct polynomials associated to set partitions. Because the shape consists of two full rows of boxes followed by one box in each additional row, we call these \emph{jellyfish tableaux} (cf.\ Example~\ref{ex:RStableaux}).
\begin{definition}
\label{def:rstab}
Given a set partition $\pi=\{\pi_1,\pi_2,\ldots,\pi_d\}\in\setpart(n,d)$, let $\rstab(\pi)$ be the set of generalized tableaux $T_{ij}$ (in English notation with matrix indexing) with $d$ columns (so $1\leq j\leq d$) and $n-2d+2$ rows ($1\leq i\leq n-2d+2$) obeying the following constraints:
\begin{enumerate}
\item $T_{ij}\in[n]$ or $T_{ij}$ is empty.
\item If $i\in\{1,2\}$, $T_{ij}$ is nonempty.
\item If $i>2$, there exists exactly one $j$ such that $T_{ij}$ is nonempty.
\item The nonempty entries in column $j$ are exactly the elements of $\pi_j$, in increasing order.
\end{enumerate}
Call $\rstab(\pi)$ the set of \newword{jellyfish tableaux} for $\pi$. (Note that the set of jellyfish tableaux depends on the chosen ordering convention for blocks of set partitions.)
\end{definition}

\begin{definition}
\label{def:inv}
Let $T\in \rstab(\pi)$. Define the \newword{inversion number} $\inv(T)$ as the number of inversions in the row reading word (left to right, top to bottom). Define the \newword{sign} of $T$ as $\sgn(T)=(-1)^{\inv(T)}$.
\end{definition}

Note that $i < j$ form an inversion of $T$ if and only if either $j$ appears in a higher row than $i$ or else $j$ appears left of $i$ in the same row. 

Recall the matrix $M$ from Section~\ref{sec:flagvarieties}. Let $I$ and $J$ be finite subsets of $\mathbb{N}$ and let $M_I^J$ denote the determinant of the submatrix of $M$ with rows indexed by $I$ and columns indexed by $J$, in increasing order. (For convenience, we sometimes write elements separated by commas in the subscript and superscript rather than formal sets.) Note since these minors may not be top-justified, they are not naturally expressed in Pl\"ucker variables.

\begin{definition}
\label{def:rsdet}
Given $\pi\in\setpart(n,d)$ and $T\in \rstab(\pi)$, define the product of determinants \[\rspoly(T)= 
\displaystyle\prod_{i=1}^d M_{R_i(T)}^{\pi_i},\] where $R_i(T)$ is the set of rows of $T$ containing an entry in $\pi_i$. 
\end{definition}

For example, below is a jellyfish tableau and its corresponding product of determinants. 
\begin{center}
\raisebox{4ex}{$T= \begin{ytableau}
2 & 5 & 1\\
3 & 7 & 4 \\
6 & \none \\
\none & 8 \\
10 & \none \\
\none & 9
\end{ytableau}$}
\hspace{.5in}
$\rspoly(T)=\begin{vmatrix}
x_{12} & x_{13} & x_{16} & x_{110}\\
x_{22} & x_{23} & x_{26} & x_{210}\\
x_{32} & x_{33} & x_{36} & x_{310}\\
x_{52} & x_{53} & x_{56} & x_{510}
\end{vmatrix}\cdot
\begin{vmatrix}
x_{15} & x_{17} & x_{18} & x_{19}\\
x_{25} & x_{27} & x_{28} & x_{29}\\
x_{45} & x_{47} & x_{48} & x_{49}\\
x_{65} & x_{67} & x_{68} & x_{69}
\end{vmatrix}\cdot
\begin{vmatrix}
x_{11} & x_{14} \\
x_{21} & x_{24}
\end{vmatrix}
$

\end{center}

Note that Definition~\ref{def:rsdet} reduces to the Pl\"ucker case when $\pi$ is a matching.

We now define a polynomial web invariant for each set partition $\pi$. Though we generally consider the case of no singletons, for the sake of our recurrence in Theorem~\ref{thm:5term}, we define the polynomial to be $0$ if the set partition has a singleton block.

\begin{definition}
\label{def:polynomial}
Given a set partition $\pi$ with no singletons, let $[\pi]$ denote the polynomial \[
[\pi] = \sum_{T\in \rstab(\pi)}\sgn(T) \; \rspoly(T).
\]
If $\pi$ has a singleton block, we set $[\pi]=0$.
\end{definition}

While Definition~\ref{def:polynomial} appears to depend on our choice of ordering convention for blocks of set partitions, we will show in Lemma~\ref{lem:row_col_swap} that the polynomial $[\pi]$ is independent of such choices.

\begin{example}\label{ex:RStableaux} Suppose $\pi=\{\{2, 3, 6, 10\},\{5,7,8,9\},\{1,4\}\}$ (where we assume the blocks are ordered as written). Then $\rstab(\pi)$ consists of the tableaux below.
\[\begin{ytableau}
2 & 5 & 1 \\
3 & 7 & 4 \\
6 \\
10\\
\none & 8 \\
\none & 9
\end{ytableau}\hspace{.3in}
\begin{ytableau}
2 & 5 & 1 \\
3 & 7 & 4 \\
6 \\
\none & 8\\
10 \\
\none & 9
\end{ytableau}\hspace{.3in}
\begin{ytableau}
2 & 5 & 1 \\
3 & 7 & 4 \\
6 \\
\none & 8\\
\none & 9\\
10
\end{ytableau}\hspace{.3in}
\begin{ytableau}
2 & 5 & 1 \\
3 & 7 & 4 \\
\none & 8\\
6 \\
10 \\
\none & 9
\end{ytableau}\hspace{.3in}
\begin{ytableau}
2 & 5 & 1 \\
3 & 7 & 4 \\
\none & 8\\
6 \\
\none & 9\\
10
\end{ytableau}\hspace{.3in}
\begin{ytableau}
2 & 5 & 1 \\
3 & 7 & 4 \\
\none & 8\\
\none & 9\\
6 \\
10
\end{ytableau}\]
The leftmost tableau has row reading word 2,5,1,3,7,4,6,10,8,9 and thus has 8 inversions. Reading the list of tableaux from left to right, the tableaux have 8, 7, 6, 8, 7, and 8 inversions, respectively. Finally, we have that 
\begin{align*}[\pi]=M_{1,2,3,4}^{2,3,6,10}\cdot M_{1,2,5,6}^{5,7,8,9}\cdot M_{1,2}^{1,4}
-M_{1,2,3,5}^{2,3,6,10}\cdot M_{1,2,4,6}^{5,7,8,9}\cdot M_{1,2}^{1,4}
+M_{1,2,3,6}^{2,3,6,10}\cdot M_{1,2,4,5}^{5,7,8,9}\cdot M_{1,2}^{1,4}\\
+M_{1,2,4,5}^{2,3,6,10}\cdot M_{1,2,3,6}^{5,7,8,9}\cdot M_{1,2}^{1,4}
-M_{1,2,4,6}^{2,3,6,10}\cdot M_{1,2,3,5}^{5,7,8,9}\cdot M_{1,2}^{1,4}
+M_{1,2,5,6}^{2,3,6,10}\cdot M_{1,2,3,4}^{5,7,8,9}\cdot M_{1,2}^{1,4}.
\end{align*}
\end{example}

\subsection{Inversion counting lemmas}
\label{sec:lemmas}

We now prove some useful lemmas on comparing the number of inversions in related tableaux.

\begin{definition}
Let $\rstabgen(T)$ be the set of tableaux obtainable from $T\in\rstab(\pi)$ by reordering the set of columns and/or permuting entries in a column. Let $\rstabgen(\pi)=\displaystyle\bigcup_{T\in\rstab(\pi)}\rstabgen(T)$.
\end{definition}

\begin{definition}
\label{def:inv2}
Let $T\in \rstabgen(\pi)$. Define the \newword{inversion number} $\inv(T)$ as the number of inversions in the row reading word (left to right, top to bottom), except we do not count inversions within columns. Define $\sgn(T)=(-1)^{\inv(T)}$.
\end{definition}

Note that $\rstab(\pi) \subsetneq \rstabgen(\pi)$ and Definition~\ref{def:inv} agrees with the restriction of Definition~\ref{def:inv2}, since there are no inversions within columns of a tableau in $\rstab(\pi)$.

\begin{example}
Consider the tableaux $U$ and $T$ shown below. Then $U$ is one of the jellyfish tableaux from Example~\ref{ex:RStableaux} and $T\in\rstabgen(U)$. The pairs of entries $(4,7)$, $(1,2)$, and $(9,10)$, for example, form inversions in $T$ while pairs like $(5,8)$ and $(6,10)$ do not since those entries lie in the same column. We observed in Example~\ref{ex:RStableaux} that $\inv(U)=7$, and the reader may check that $\inv(T)=15$. While $\inv(U)\neq\inv(T)$, we see that $\sgn(U)=\sgn(T)$. We explore this phenomenon further in Lemma~\ref{lem:row_col_swap}.

\[U=\begin{ytableau}
2 & 5 & 1 \\
3 & 7 & 4 \\
6 \\
\none & 8\\
10 \\
\none & 9
\end{ytableau}\hspace{1in}
T=\begin{ytableau}
7 & 4 & 2 \\
8 & 1 & 10\\
\none & \none & 6\\
 5\\
\none & \none & 3\\
9
\end{ytableau}\]
\end{example}

The following two lemmas will be used in the proof of our main recurrence, Theorem~\ref{thm:5term}. Lemma~\ref{lem:row_col_swap} also shows that Definition~\ref{def:polynomial} does not depend on the ordering convention for the blocks of $\pi$.
\begin{lemma}
\label{lem:row_col_swap}
Suppose $U\in \rstab(\pi)$ and $T\in \rstabgen(U)$. Then $\sgn(U)=\sgn(T)$.
\end{lemma}

\begin{proof}
We first address the case of permuted labels within a column.
Suppose $T\in \rstabgen(U)$ and suppose $T'$ is obtained from $T$ by permuting the labels within columns.
We show $\sgn(T)=\sgn(T')$ (and hence, by repeated application, $\sgn(T) = \sgn(U)$). Note $\inv(T)$ may not equal $\inv(T')$, but they have the same parity.

It is enough to prove the result in the case that $T'$ is obtained from $T$ by interchanging two entries $i$ and $j$ in a single column.
Recall that $\inv(T)$ is the number of inversions in the row reading word of $T$, except ignoring inversions within columns. For each pair of entries $k,\ell$ of $T$, define
\[
\iota(k,\ell) = \begin{cases}
-1, & \text{if $k$ and $\ell$ are inverted in the reading word of $T$;} \\
1, & \text{otherwise.}\\
\end{cases}
\]

Consider any entry $k$ outside of the column containing $i$ and $j$.
Interchanging the elements $i$ and $j$ in $T$ swaps their
positions in the reading word. Therefore, the sum $\iota(i,k) + \iota(j,k) = 0$. Hence, this swap does not affect the parity of the number of inversions involving $(i, k)$ and $(j, k)$.
It does change whether $(i, j)$ is an inversion in the reading word; however, since they are in the same column, this does not affect the inversion number of the jellyfish tableau. Hence, $\sgn(T) = \sgn(T')$.

We now address the case of permuting the set of columns.
Suppose $T''$ is obtained from $T'$ by permuting columns $i$ and $j$. Inversions in $T'$ are the same as inversion in $T''$ except for inversions between entries in the first row of columns $i$ and $j$ and inversions between entries in the second row and in columns $i$ and $j$. If there was an inversion between the pair of entries in the first (resp.\ second) row of $T'$, there will not be an inversion between the pair of entries in the first (resp.\ second) row of $T''$. If there was not an inversion between the pair of entries in the first (resp.\ second) row of $T'$, there will be an inversion between the pair of entries in the first (resp.\ second) row of $T''$. This will not change the parity of the number of inversions. 

We conclude that ${\sgn(T')}={\sgn(T'')}$.
\end{proof}

\begin{lemma}
\label{lem:sgn_ell}
Suppose $n>4$, $\pi\in\setpart(n,d)$ with $d<n/2$ (so that $\pi$ is not a matching), and $T\in\rstab(\pi)$. 
Suppose $\pi_{m}$ is the unique part of $\pi$ with entry in the lowest row of $T$, and let $k\in\pi_m$. Suppose $T_{k}$ is the tableau constructed by removing $k$, shifting all numbers in that column up to fill the position vacated by $k$, and thus removing the position in the last row, leaving all other columns fixed. Then $\sgn(T)=\sgn(T_{k})(-1)^{n-k-|\pi_{m}|_{>k}}$, where $|\pi_m|_{>k}$ denotes the number of elements in $\pi_m$ greater than $k$.
\end{lemma}
\begin{proof}
Suppose the conditions of the lemma hold. Then $\sgn(T)=\sgn(T_{k})(-1)^{\alpha}$ where $\alpha$ is the difference in the number of inversions between $T$ and $T_{k}$. By Lemma~\ref{lem:row_col_swap}, we have $\sgn(T)=\sgn(T')$, where $T'$ is the tableau in $\rstabgen(T)$ in which $k$ is moved to the last row of its column and the other entries are shifted up to fill the position left by $k$. Then $T'$ differs from $T_{k}$ by deleting $k$, so we need only count the inversions involving $k$. Entry $k$ has inversions with every number larger than $k$ that is not in its column, but since it is in the last row, it has no inversions with numbers smaller than it. There are $n-k$ numbers larger than $k$ in $T$. 
Removing $k$ from $T'$ removes the inversions between $k$ and each larger number in a different column, so $\alpha={n-k-|\pi_m|_{>k}}$. Therefore, $\sgn(T)=\sgn(T_{k})(-1)^{n-k-|\pi_m|_{>k}}$.
\end{proof}

\subsection{A five-term recurrence}
\label{sec:5term}
In this subsection, we state and prove our main recurrence, Theorem~\ref{thm:5term}, which we use in the next subsection to prove Theorem~\ref{thm:main}. In particular, this recurrence will be our main tool for establishing Lemma~\ref{lemma:invariants_pennant}, showing that the web invariants $[\pi]$ live in the appropriate pennant Specht modules.

\begin{figure}[htbp]
\begin{center}
\includegraphics[width=\textwidth]{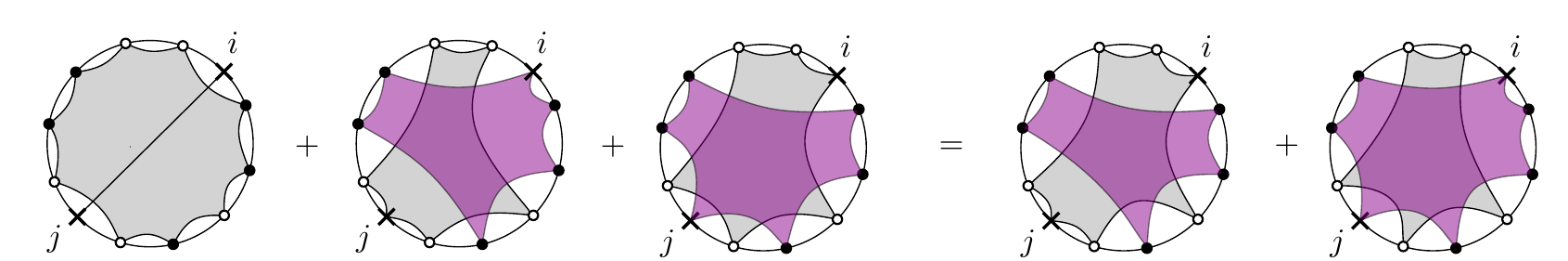}
\caption{A schematic illustration of the five-term recurrence of Theorem~\ref{thm:5term}.}
\label{fig:fivetermrecurrence}
\end{center}
\end{figure}

\begin{theorem}
\label{thm:5term}
Partition $\{1,\ldots n\}$ into four nonempty sets: $A$, $B$, $I$, and $J$, where $|I|=|J|=1$. Then 
\begin{equation}
\label{eq:5term}
\Big[\{A\cup B,I\cup J\}\Big] + \Big[\{A\cup I,B\cup J\}\Big]+ \Big[\{A\cup J,B\cup I\}\Big] = \Big[\{A\cup I\cup J,B\}\Big]+ \Big[\{A,B\cup I\cup J\}\Big].
\end{equation}
\end{theorem}

The proof of this theorem, given below, is rather technical, but elementary. It follows by noting that the leftmost polynomial is associated to a single jellyfish tableau representing a $2\times 2$ determinant multiplied by an $(n-2)\times(n-2)$ determinant. The proof proceeds by row expansion of the larger determinant and induction. We use Lemmas~\ref{lem:row_col_swap} and~\ref{lem:sgn_ell} in several key spots to determine the signs of various jellyfish tableaux.

\begin{proof}
Note that we will often use sets as lists, assuming the ordering on the list corresponding to a set is increasing. We use the notation $[b]:=\{1,2,\ldots,b\}$.

We proceed by induction on $n$. The base case is $n=4$. In this case, $|A|=|B|=1$, so the right hand side terms are $0$ since they each have a singleton. This then reduces to the a classical determinantal identity, the Pl\"ucker relation.

Suppose $n>4$ and that the identity holds for set partitions of any $n-1$ elements in $[n]$. Partition $[n]$ into four nonempty sets: $A$, $B$, $I=\{i\}$, and $J=\{j\}$. For simplicity, we assume $i<j$.

We begin by noting that there is a unique jellyfish tableau $T=\scalebox{.7}{\ytableaushort{{~}i,{~}j,{~},{\vdots},{~}}}$ in $\rstab(\{A\cup B,I\cup J\})$, where the first column is filled with the elements of $A\cup B$ and the second column contains $i$ and $j$. This is because $\{A\cup B,I\cup J\}$ has one part of size $2$ (and those tableau entries must be in rows $1$ and $2$) and another part of size $n-2$.  So the polynomial $\Big[\{A\cup B,I\cup J\}\Big]=\sgn({T})M_{[n-2]}^{A\cup B} M_{1,2}^{i,j}$ is exactly one term.

We claim that $\sgn({T})=(-1)^{i+j}$ and show this by cases. In all cases, we assume the columns are ordered as written, since by Lemma~\ref{lem:row_col_swap}, $\sgn(T)$ is invariant under reordering the set of columns.

\textit{Case $i\geq 2$:} In this case, $i$ has an inversion with all numbers less than $i$ except $1$, which occurs to the left of $i$ in the first row. So the contribution to $\sgn(T)$ is $(-1)^{i-2}=(-1)^i$. If $j\geq 4$, $j$ has an inversion with all numbers less than it except for $1, i,$ and the number to its left (which is $3$ if $i=2$ and $2$ otherwise). So the contribution to $\sgn(T)$ is $(-1)^{j-4}=(-1)^j$. If $j=3$, this means $i=2$, and $j$ has one inversion with $3$ which is to its left. So the contribution to $\sgn(T)$ is $(-1)^{1}=(-1)^j$ since $j=3$.

\textit{Case $i = 1$:} In this case, $i$ has an inversion with $2$ which is to its left. So the contribution to $\sgn(T)$ is $(-1)^{1}=(-1)^i$ since $i=1$. If $j\geq 4$, $j$ has an inversion with all numbers less than it except for $2, 3$, and $i$. So the contribution to $\sgn(T)$ is $(-1)^{j-4}=(-1)^j$. If $j=3$, $j$ has one inversion with $4$ which is to its left. So the contribution to $\sgn(T)$ is $(-1)^{1}=(-1)^j$ since $j=3$. If $j=2$, $j$ has inversions with $3$ and $4$.  So the contribution to $\sgn(T)$ is $(-1)^{2}=(-1)^j$ since $j=2$.

In all cases, $\sgn({T})=(-1)^{i+j}$.

We expand the first determinant along the bottom row, obtaining the following. Note we account for skipping the columns associated to $i$ and $j$ with the minus sign outside the second summation.
\begin{align*}
\Big[\{A\cup B,I\cup J\}\Big]
=& (-1)^{i+j}M_{[n-2]}^{A\cup B} M_{1,2}^{I\cup J}\\
=& (-1)^{i+j}\Bigg(\sum_{k=1}^{i-1} (-1)^{n-k}x_{n-2,k} M_{[n-3]}^{(A\cup B)\setminus k}M_{1,2}^{I\cup J}\\
&-\sum_{k=i+1}^{j-1} (-1)^{n-k}x_{n-2,k} M_{[n-3]}^{(A\cup B)\setminus k}M_{1,2}^{I\cup J}\\
&{+\sum_{k=j+1}^{n} (-1)^{n-k}x_{n-2,k} M_{[n-3]}^{(A\cup B)\setminus k}M_{1,2}^{I\cup J}\Bigg)}
\end{align*}

We then note that $M_{[n-3]}^{(A\cup B)\setminus k}M_{1,2}^{i,j}$ is up to sign the polynomial $\Big[\{(A\cup B)\setminus k,I\cup J\}\Big]$. Let ${T}_{k}$ be the unique jellyfish tableau in $\rstab(\{(A\cup B)\setminus k,I\cup J\})$. Then we claim $\sgn(T_{k})=(-1)^{i+j}$ if $k<i$ {or $k>j$} and $\sgn(T_{k})=(-1)^{i+j-1}$ if {$i<k<j$}.

We consider $T_{k}$ in comparison to $T$. Use Lemma~\ref{lem:row_col_swap} to conclude that $\sgn(T)$ equals the sign of the tableau $\tilde{T}_{k}$ where $\tilde{T}_{k}$ is defined as the tableau in $\rstabgen(T)$ in which $k$ is shifted to be the last number in its column.

\textit{Case $k<i$:}
If $k<i$, $k$ has an inversion with both $i$ and $j$ in $\tilde{T}_{k}$. Deleting $k$ removes two inversions, and thus $\sgn(T_{k})=\sgn(T)=(-1)^{i+j}$.

\textit{Case $i<k<j$:}
If $i<k<j$, $k$ does not have an inversion with $i$ in $\tilde{T}_{k}$, but it does have an inversion with $j$. So deleting $k$ removes one inversion. Thus $\sgn(T_{k})=\sgn(T)(-1)=(-1)^{i+j-1}$.

\textit{Case $k>j$:}
If $k>j$, $k$ does not have an inversion with $i$ or $j$ in $\tilde{T}_{k}$. Deleting $k$ does not change the inversion number, and so $\sgn(T_{k})=\sgn(T)=(-1)^{i+j}$.

Therefore we have 
\begin{align*}
\Big[\{A\cup B,I\cup J\}\Big]
=(-1)^{i+j}&\Bigg(\sum_{k=1}^{i-1} (-1)^{n-k}x_{n-2,k} M_{[n-3]}^{(A\cup B)\setminus k}M_{1,2}^{I\cup J}\\
&-\sum_{k=i+1}^{j-1} (-1)^{n-k}x_{n-2,k} M_{[n-3]}^{(A\cup B)\setminus k}M_{1,2}^{I\cup J}\\
&+\sum_{k=j+1}^{n} (-1)^{n-k}x_{n-2,k} M_{[n-3]}^{(A\cup B)\setminus k}M_{1,2}^{I\cup J}\Bigg)\\
=(-1)^{i+j}&\Bigg(\sum_{k=1}^{i-1} (-1)^{n-k}x_{n-2,k} (-1)^{i+j}\Big[\{(A\cup B)\setminus k,I\cup J\}\Big]\\
&+\sum_{k=i+1}^{j-1} (-1)^{n-k-1}x_{n-2,k} (-1)^{i+j-1}\Big[\{(A\cup B)\setminus k,I\cup J\}\Big]\\
&+\sum_{k=j+1}^{n} (-1)^{n-k}x_{n-2,k} (-1)^{i+j}\Big[\{(A\cup B)\setminus k,I\cup J\}\Big]\Bigg)\\
= \sum_{k\in A\cup B} (&-1)^{n-k}x_{n-2,k} \Big[\{(A\cup B)\setminus k,I\cup J\}\Big].
\end{align*}
The set partition $\{(A\cup B)\setminus k,I\cup J\}$ partitions a set with $n-1$ elements, so we may apply the induction hypothesis:
\begin{align*}
\Big[\{(A\cup B)\setminus k,I\cup J\}\Big] = &- \Big[\{A\cup I,B\cup J\}\setminus k\Big] - \Big[\{A\cup J,B\cup I\}\setminus k\Big]\\
&+ \Big[\{A\cup I\cup J,B\}\setminus k\Big]+ \Big[\{A,B\cup I\cup J\}\setminus k\Big].
\end{align*}
Note we write these set partitions as $\{\pi\}\setminus k$ since $k$ may be in either $A$ or $B$.

Now use Definitions~\ref{def:rstab} and \ref{def:rsdet} to obtain:
\begin{align*}
\Big[\{&(A\cup B)\setminus k,I\cup J\}\Big] \\
= &- \displaystyle\sum_{T\in \rstab(\{A\cup I,B\cup J\}\setminus k)}\sgn(T)M_{R_1(T)}^{(A\cup I)\setminus k}M_{R_2(T)}^{(B\cup J)\setminus k} 
- \displaystyle\sum_{T\in \rstab(\{A\cup J,B\cup I\}\setminus k)}\sgn(T)M_{R_1(T)}^{(A\cup J)\setminus k}M_{R_2(T)}^{(B\cup I)\setminus k}\\
&+ \displaystyle\sum_{T\in \rstab(\{A\cup I\cup J,B\}\setminus k)}\sgn(T)M_{R_1(T)}^{(A\cup I\cup J)\setminus k}M_{R_2(T)}^{B\setminus k}
+ \displaystyle\sum_{T\in \rstab(\{A,B\cup I\cup J\}\setminus k)}\sgn(T)M_{R_1(T)}^{A\setminus k}M_{R_2(T)}^{(B\cup I\cup J)\setminus k}.
\end{align*}

Putting this all together, we have:
\begin{align*}
\Big[\{&A\cup B,I\cup J\}\Big]\\
=& \sum_{k\in A\cup B} (-1)^{n-k}x_{n-2,k}
\Bigg(- \displaystyle\sum_{T\in \rstab(\{A\cup I,B\cup J\}\setminus k)}\sgn(T)M_{R_1(T)}^{(A\cup I)\setminus k}M_{R_2(T)}^{(B\cup J)\setminus k} \\
&- \displaystyle\sum_{T\in \rstab(\{A\cup J,B\cup I\}\setminus k)}\sgn(T)M_{R_1(T)}^{(A\cup J)\setminus k}M_{R_2(T)}^{(B\cup I)\setminus k}
+ \displaystyle\sum_{T\in \rstab(\{A\cup I\cup J,B\}\setminus k)}\sgn(T)M_{R_1(T)}^{(A\cup I\cup J)\setminus k}M_{R_2(T)}^{B\setminus k}\\
&+ \displaystyle\sum_{T\in \rstab(\{A,B\cup I\cup J\}\setminus k)}\sgn(T)M_{R_1(T)}^{A\setminus k}M_{R_2(T)}^{(B\cup I\cup J)\setminus k}\Bigg).
\end{align*}

We now split the outer summation according to whether $k$ is in $A$ or $B$.
\begin{align*}
\Big[\{&A\cup B,I\cup J\}\Big]\\
=& \sum_{k\in A} (-1)^{n-k}x_{n-2,k}
\Bigg(- \displaystyle\sum_{T\in \rstab(\{(A\cup I)\setminus k,B\cup J\})}\sgn(T)M_{R_1(T)}^{(A\cup I)\setminus k}M_{R_2(T)}^{B\cup J}\\ 
&- \displaystyle\sum_{T\in \rstab(\{(A\cup J)\setminus k,B\cup I\})}\sgn(T)M_{R_1(T)}^{(A\cup J)\setminus k}M_{R_2(T)}^{B\cup I}\\
&+ \displaystyle\sum_{T\in \rstab(\{(A\cup I\cup J)\setminus k,B\})}\sgn(T)M_{R_1(T)}^{(A\cup I\cup J)\setminus k}M_{R_2(T)}^{B}\\
&+ \displaystyle\sum_{T\in \rstab(\{A\setminus k,B\cup I\cup J\})}\sgn(T)M_{R_1(T)}^{A\setminus k}M_{R_2(T)}^{B\cup I\cup J}\Bigg)\\
&+ \sum_{k\in B} (-1)^{n-k}x_{n-2,k}
\Bigg(- \displaystyle\sum_{T\in \rstab(\{A\cup I,(B\cup J)\setminus k\})}\sgn(T)M_{R_1(T)}^{A\cup I}M_{R_2(T)}^{(B\cup J)\setminus k}\\ 
&- \displaystyle\sum_{T\in \rstab(\{A\cup J,(B\cup I)\setminus k\})}\sgn(T)M_{R_1(T)}^{A\cup J}M_{R_2(T)}^{(B\cup I)\setminus k}\\
&+ \displaystyle\sum_{T\in \rstab(\{A\cup I\cup J,B\setminus k\})}\sgn(T)M_{R_1(T)}^{A\cup I\cup J}M_{R_2(T)}^{B\setminus k}\\
&+ \displaystyle\sum_{T\in \rstab(\{A,(B\cup I\cup J)\setminus k\})}\sgn(T)M_{R_1(T)}^{A}M_{R_2(T)}^{(B\cup I\cup J)\setminus k}\Bigg)
\end{align*}

Note that there is a natural bijection between, e.g., $\rstab(\{(A\cup I)\setminus k,B\cup J\})$ and the set of tableaux in $\rstab(\{(A\cup I),B\cup J\})$ such that the element in row $n-2$ is in column $1$. 
Let $T$ be a tableau in the second set described in the previous sentence.
Then the corresponding tableau $T_{k}\in\rstab(\{(A\cup I)\setminus k,B\cup J\})$ is constructed by deleting entry $k$ and position $n-2$ from column $1$ of $T$ and shifting entries up to fill the position left by $k$. 
By Lemma~\ref{lem:sgn_ell}, $\sgn(T)=\sgn(T_{k})(-1)^{n-k-|A\cup I|_{>k}}$. 
Note that the $(-1)^{n-k}$ in the previous sentence cancels with the $(-1)^{n-k}$ in front of the large parentheses.

This bijection is valid for any $k$, so this allows us to sum over the set of tableaux in $\rstab(\{(A\cup I),B\cup J\})$ such that the element in row $n-2$ is in column $1$ rather than $\rstab(\{(A\cup I)\setminus k,B\cup J\})$, as long as we adjust the signs using Lemma~\ref{lem:sgn_ell}. This allows us to interchange the order of summation.

Let $\rstab_{k}(\pi)$ denote the set of tableaux in $\rstab(\pi)$ with the property that the element the last row is in column $k$. We then have the following:

\begin{align*}
\Big[\{A\cup B&,I\cup J\}\Big]\\
=& 
- \displaystyle\sum_{T\in \rstab_1(\{A\cup I,B\cup J\})}\sgn(T)\sum_{k\in A} (-1)^{|A\cup I|_{>k}}x_{n-2,k}M_{R_1(T)\setminus\{n-2\}}^{(A\cup I)\setminus k}M_{R_2(T)}^{B\cup J}\\ 
&- \displaystyle\sum_{T\in \rstab_1(\{A\cup J,B\cup I\})}\sgn(T)\sum_{k\in A} (-1)^{|A\cup J|_{>k}}x_{n-2,k}M_{R_1(T)\setminus\{n-2\}}^{(A\cup J)\setminus k}M_{R_2(T)}^{B\cup I}\\
&+ \displaystyle\sum_{T\in \rstab_1(\{A\cup I\cup J,B\})}\sgn(T)\sum_{k\in A} (-1)^{|A\cup I\cup J|_{>k}}x_{n-2,k}M_{R_1(T)\setminus\{n-2\}}^{(A\cup I\cup J)\setminus k}M_{R_2(T)}^{B}\\
&+ \displaystyle\sum_{T\in \rstab_1(\{A,B\cup I\cup J\})}\sgn(T)\sum_{k\in A} (-1)^{|A|_{>k}}x_{n-2,k}M_{R_1(T)\setminus\{n-2\}}^{A\setminus k}M_{R_2(T)}^{B\cup I\cup J}\\
& 
- \displaystyle\sum_{T\in \rstab_2(\{A\cup I,B\cup J\})}\sgn(T)\sum_{k\in B} (-1)^{|B\cup J|_{>k}}x_{n-2,k}M_{R_1(T)}^{A\cup I}M_{R_2(T)\setminus\{n-2\}}^{(B\cup J)\setminus k}\\ 
&- \displaystyle\sum_{T\in \rstab_2(\{A\cup J,B\cup I\})}\sgn(T)\sum_{k\in B} (-1)^{|B\cup I|_{>k}}x_{n-2,k}M_{R_1(T)}^{A\cup J}M_{R_2(T)\setminus\{n-2\}}^{(B\cup I)\setminus k}\\
&+ \displaystyle\sum_{T\in \rstab_2(\{A\cup I\cup J,B\})}\sgn(T)\sum_{k\in B} (-1)^{|B|_{>k}}x_{n-2,k}M_{R_1(T)}^{A\cup I\cup J}M_{R_2(T)\setminus\{n-2\}}^{B\setminus k}\\
&+ \displaystyle\sum_{T\in \rstab_2(\{A,B\cup I\cup J\})}\sgn(T)\sum_{k\in B} (-1)^{|B\cup I\cup J|_{>k}}x_{n-2,k}M_{R_1(T)}^{A}M_{R_2(T)\setminus\{n-2\}}^{(B\cup I\cup J)\setminus k}.
\end{align*}

Note that the second summation of each line is \emph{almost} a bottom row determinant expansion associated to a matrix of one size larger, except the term(s) corresponding to $I$ and/or $J$ are not included. We add and subtract the needed terms as follows.

We include full details for the second summand on the first line. First, we add and subtract the $x_{n-2,i}$ term with the needed sign. Then we recollect as a larger determinant and one remaining term. 
\begin{align*}
\sum_{k\in A} (-1)^{|A\cup I|_{>k}}&x_{n-2,k}M_{R_1(T)\setminus\{n-2\}}^{(A\cup I)\setminus k}\\
=& \sum_{k\in A} (-1)^{|A\cup I|_{>k}}x_{n-2,k}M_{R_1(T)\setminus\{n-2\}}^{(A\cup I)\setminus k}\\
&+(-1)^{|A\cup I|_{>i}}x_{n-2,i}M_{R_1(T)\setminus\{n-2\}}^A-(-1)^{|A\cup I|_{>i}}x_{n-2,i}M_{R_1(T)\setminus\{n-2\}}^A\\
=& M_{R_1(T)}^{A\cup I}-(-1)^{|A\cup I|_{>i}}x_{n-2,i}M_{R_1(T)\setminus\{n-2\}}^A\\
\end{align*}

We do this for each line, except lines 4 and 7 need no additional terms added and subtracted, while lines 3 and 8 need two additional terms added and subtracted.

We then have:
\begin{align*}
\Big[\{&A\cup B,I\cup J\}\Big]\\
=& 
- \displaystyle\sum_{T\in \rstab_1(\{A\cup I,B\cup J\}) }\sgn(T)\Big(M_{R_1(T)}^{A\cup I}-(-1)^{|A\cup I|_{>i}}x_{n-2,i}M_{R_1(T)\setminus\{n-2\}}^A \Big)M_{R_2(T)}^{B\cup J}\\ 
&- \displaystyle\sum_{T\in \rstab_1(\{A\cup J,B\cup I\})}\sgn(T)\Big(M_{R_1(T)}^{A\cup J}-(-1)^{|A\cup J|_{>j}}x_{n-2,j}M_{R_1(T)\setminus\{n-2\}}^A \Big)M_{R_2(T)}^{B\cup I}\\
&+ \displaystyle\sum_{T\in \rstab_1(\{A\cup I\cup J,B\})}\sgn(T)\\
&\Big(M_{R_1(T)}^{A\cup I\cup J}-(-1)^{|A\cup I\cup J|_{>i}}x_{n-2,i}M_{R_1(T)\setminus\{n-2\}}^{A\cup J} -(-1)^{|A\cup I\cup J|_{>j}}x_{n-2,j}M_{R_1(T)\setminus\{n-2\}}^{A\cup I}\Big)M_{R_2(T)}^{B}\\
&+ \displaystyle\sum_{T\in \rstab_1(\{A,B\cup I\cup J\})}\sgn(T)\Big(M_{R_1(T)}^{A} \Big)M_{R_2(T)}^{B\cup I\cup J}\\
&
- \displaystyle\sum_{T\in \rstab_2(\{A\cup I,B\cup J\})}\sgn(T)\Big(M_{R_2(T)}^{B\cup J}-(-1)^{|B\cup J|_{>j}}x_{n-2,j}M_{R_2(T)\setminus\{n-2\}}^B \Big)M_{R_1(T)}^{A\cup I}\\ 
&- \displaystyle\sum_{T\in \rstab_2(\{A\cup J,B\cup I\})}\sgn(T)\Big(M_{R_2(T)}^{B\cup I}-(-1)^{|B\cup I|_{>i}}x_{n-2,i}M_{R_2(T)\setminus\{n-2\}}^B \Big)M_{R_1(T)}^{A\cup J}\\
&+ \displaystyle\sum_{T\in \rstab_2(\{A\cup I\cup J,B\})}\sgn(T)\Big(M_{R_2(T)}^B\Big)M_{R_1(T)}^{A\cup I\cup J}\\
&+ \displaystyle\sum_{T\in \rstab_2(\{A,B\cup I\cup J\})}\sgn(T)\\
&\Big(M_{R_2(T)}^{B\cup I\cup J}-(-1)^{|B\cup I\cup J|_{>i}}x_{n-2,i}M_{R_2(T)\setminus\{n-2\}}^{B\cup J}  -(-1)^{|B\cup I\cup J|_{>j}}x_{n-2,j}M_{R_2(T)\setminus\{n-2\}}^{B\cup I} \Big)M_{R_1(T)}^{A}.
\end{align*}

We now note there are four pairs of terms in the above that match, so if the signs are opposite, they will cancel. We show this is the case by analyzing the signs.

First let $T\in\rstab_1(\{A\cup I,B\cup J\})$. Let $T'\in\rstab_2(\{A,(B\cup I\cup J))$ be the unique tableau such that $R_1(T)\setminus\{n-2\}=R_1(T')$ and $R_2(T)=R_2(T')\setminus\{n-2\}$. We compare $\sgn(T)$ and $\sgn(T')$. One difference between $T$ and $T'$ is that $T$ has $i$ in column $1$, while $T'$ has $i$ in column $2$. The other difference is that in $T$, the unique entry in row $n-2$ is in column $1$, while in $T'$, it is in column $2$.  By Lemma~\ref{lem:row_col_swap}, $\sgn(T)$ equals the sign of the tableau, denoted $T_i$, in which $i$ is moved to the last entry of its column. Similarly, $\sgn(T')$ equals the sign of the tableau, denoted $T_i'$, in which $i$ is moved to the last entry of its column. Moreover, by Lemma~\ref{lem:sgn_ell}, 
\[\sgn(T)=\sgn(T_i)(-1)^{n-i-|A\cup I|_{>i}}  \text{ and } \sgn(T')=\sgn(T_i')(-1)^{n-i-|B\cup I\cup J|_{>i}}.\] But $T_i=T_i'$. So 
\[\sgn(T)(-1)^{n-i-|A\cup I|_{>i}}=\sgn(T')(-1)^{n-i-|B\cup I\cup J|_{>i}},\] and more simply, 
\[\sgn(T)(-1)^{|A\cup I|_{>i}}=\sgn(T')(-1)^{|B\cup I\cup J|_{>i}}.\] Therefore, the second term of the first sum cancels with the second term of the eighth sum. By similar reasoning, the second term of the second sum cancels with the third term of the eighth sum, the second term of the third sum cancels with the second term of the sixth sum, and the third term of the third sum cancels with the second term of the fifth sum.

So we have:
\begin{align*}
\Big[\{&A\cup B,I\cup J\}\Big]\\
=& 
- \displaystyle\sum_{T\in \rstab_1(\{A\cup I,B\cup J\})}\sgn(T)M_{R_2(T)}^{B\cup J}M_{R_1(T)}^{A\cup I}
- \displaystyle\sum_{T\in \rstab_1(\{A\cup J,B\cup I\})}\sgn(T)M_{R_2(T)}^{B\cup I} M_{R_1(T)}^{A\cup J}\\
&+ \displaystyle\sum_{T\in \rstab_1(\{A\cup I\cup J,B\})}\sgn(T)M_{R_2(T)}^{B}M_{R_1(T)}^{A\cup I\cup J}
+ \displaystyle\sum_{T\in \rstab_1(\{A,B\cup I\cup J\})}\sgn(T)M_{R_2(T)}^{B\cup I\cup J}M_{R_1(T)}^{A}\\
& 
- \displaystyle\sum_{T\in \rstab_2(\{A\cup I,B\cup J\})}\sgn(T)M_{R_1(T)}^{A\cup I} M_{R_2(T)}^{B\cup J} 
- \displaystyle\sum_{T\in \rstab_2(\{A\cup J,B\cup I\})}\sgn(T)M_{R_1(T)}^{A\cup J} M_{R_2(T)}^{B\cup I}\\
&+ \displaystyle\sum_{T\in \rstab_2(\{A\cup I\cup J,B\})}\sgn(T)M_{R_1(T)}^{A\cup I\cup J}M_{R_2(T)}^{B}
+ \displaystyle\sum_{T\in \rstab_2(\{A,B\cup I\cup J\})}\sgn(T)M_{R_1(T)}^{A}M_{R_2(T)}^{B\cup I\cup J}
.
\end{align*}

Noting that $\rstab_1(\pi)\cup\rstab_2(\pi)=\rstab(\pi)$ for any set partition $\pi$ with two parts, we combine sums, obtaining:
\begin{align*}
\Big[\{&A\cup B,I\cup J\}\Big]\\
=& - \displaystyle\sum_{T\in \rstab(\{A\cup I,B\cup J\})}\sgn(T)M_{R_1(T)}^{A\cup I}M_{R_2(T)}^{B\cup J} 
- \displaystyle\sum_{T\in \rstab(\{A\cup J,B\cup I\})}\sgn(T)M_{R_1(T)}^{A\cup J}M_{R_2(T)}^{B\cup I} \\
&+ \displaystyle\sum_{T\in \rstab(\{A\cup I\cup J,B\})}\sgn(T)M_{R_1(T)}^{A\cup I\cup J}M_{R_2(T)}^{B}
+ \displaystyle\sum_{T\in \rstab(\{A,B\cup I\cup J\})}\sgn(T)M_{R_1(T)}^{A}M_{R_2(T)}^{B\cup I\cup J}\\
=& -\big[\{A\cup I,B\cup J\}\big] -\big[\{A\cup J,B\cup I\}\big] + \big[\{A\cup I\cup J,B\}\big]+ \big[\{A,B\cup I\cup J\}\big],
\end{align*}
which is the desired identity.
\end{proof}

We now note that we can use this rule in set partitions with more than two blocks by fixing all blocks except those involving $A$, $B$, $I$, and $J$.
\begin{corollary}\label{cor:twoblocks}
Partition $\{1,\ldots,n\}$ into $d+2$ nonempty sets: $\pi_1,\pi_2,\ldots,\pi_{d-2}, A, B, I$, and $J$, where $|I|=|J|=1$. Then 
\begin{align*}
\big[\{\pi_1,\ldots,\pi_{d-2},A\cup B,I\cup J\}\big] +& \big[\{\pi_1,\ldots,\pi_{d-2},A\cup I,B\cup J\}\big]
+ \big[\{\pi_1,\ldots,\pi_{d-2},A\cup J,B\cup I\}\big] \\
=& \big[\{\pi_1,\ldots,\pi_{d-2},A\cup I\cup J,B\}\big]+ \big[\{\pi_1,\ldots,\pi_{d-2},A,B\cup I\cup J\}\big].
\end{align*}
\end{corollary}
\begin{proof}
If $|A|=|B|=1$, then the righthand side terms are $0$ since they each have a singleton. All the jellyfish tableaux for each of the three terms are the same shape; the only thing that differs is the content of the two length-two columns involving $A,B,I,J$. Thus, the determinants corresponding to the $\pi_1,\ldots,\pi_{d-2}$ factor out and we are left with the classical Pl\"ucker relation of determinants. 

In cases where at least one of $|A|$ or $|B|$ is greater than 1, the arguments in the proof of Theorem~\ref{thm:5term} apply, since the determinants corresponding to $\pi_1,\pi_2,\ldots,\pi_{d-2}$ factor out as long as the signs work out correctly. At each point of the proof where signs matter, we are only comparing $T$ to $T_{k}$ as in Lemma~\ref{lem:sgn_ell}. This lemma works for all jellyfish tableaux, not only those with two parts, therefore, the statement holds.
\end{proof}

\subsection{Proof of Theorem~\ref{thm:main}}
\label{sec:rest}

We now prove our main theorem via the following lemmas and propositions. We first show linear independence in Lemma~\ref{lemma:independent}. We then show in Lemma~\ref{lemma:invariants_pennant} that each polynomial associated to a set partition lives in the pennant Specht module. In Lemma~\ref{lemma:basis}, we show this subset is a basis. Finally, we prove the claims about the $S_n$ action in Corollary~\ref{prop:rotation} and note these results hold for general (not necessarily noncrossing) set partitions.

\begin{lemma}\label{lemma:independent}
The set $\{ [\pi] : \pi \in \nc(n,d) \}$ of invariants of noncrossing partitions of $n$ with $d$ blocks and no singletons is linearly independent. 
\end{lemma}
\begin{proof}
Let $k = n - 2d + 2$, and let $\pi\in\nc(n,d)$. We may order the monomials in $[\pi]$ under the lexicographic order with 
\begin{align*}x_{1,1} > x_{1,2} &> \cdots > x_{1,n} > \\
x_{2,n} > x_{2,n-1}&> \cdots >x_{2,1}>\\
x_{3,1}>x_{3,2}&>\cdots>x_{3,n}>\\
\vdots\\
x_{k,1}>x_{k,2} &>\cdots > x_{k,n}.
\end{align*}
(Note that the second row here is ordered differently from the others.)

Each element of the set $\{ [\pi] : \pi \in \nc(n,d) \}$ is a polynomial where each term has degree $n$ and has a different leading monomial under this term order. The reason is as follows. The factors of the form $x_{1,i}$ in the leading monomial come from the smallest starting point of each block, and the factors of the form $x_{2,i}$ come from the largest ending point of each block. This information is enough to determine a noncrossing partition uniquely.

Suppose $\{\pi^1,\ldots,\pi^m\}\subset \nc(n,d)$ and \[0=a_1[\pi^1]+\cdots+a_m[\pi^m].\]
One of these $[\pi^i]$ has the largest leading monomial under the lexicographic ordering, and none of the other invariants contain this monomial. It follows inductively that $a_1=\dots=a_m=0$.
\end{proof}

To prove our main theorem, we only need the following lemma in the case of noncrossing partitions. However, we prove it in generality for future use.

\begin{lemma}
\label{lemma:invariants_pennant}
The set $\{ [\pi] : \pi \in \setpart(n,d) \}$ of invariants of partitions of $n$ with $d$ blocks and no singletons is a subset of the pennant Specht module $S^{(d,d,1^{n-2d})}$.
\end{lemma}
\begin{proof}
Recall from Section~\ref{sec:flagvarieties} that $S^{(d,d,1^{n-2d})}$ is generated by fillings of partition shape $(d,d,n-2d)$ where each of $\{1,\ldots,n\}$ is used exactly once and the invariant is given by multiplying the Pl\"ucker variables corresponding to the columns of the filling. Note that if $ \pi \in \setpart(n,d)$ has $d-1$ blocks of size $2$, then $[\pi]$ is a single term where the corresponding jellyfish tableau is a filling of partition shape $(d,d,n-2d)$ with each of $\{1,\ldots,n\}$ used exactly once. Thus clearly $[\pi]\in S^{(d,d,1^{n-2d})}$. We use this fact repeatedly in our argument below.

We first consider the case where $d=2$. That is, we show that the invariant of any set partition of $n$ with 2 blocks and no singletons is in $S^{(2,2,1^{n-4})}$. We show this by induction on the number of elements in the smaller block.

The first non-trivial case is when the smaller block has three elements. Suppose that $\pi=\{\{t_1, t_2, t_3\}, \{s_1,\ldots,s_r\}\}$ and $r\geq 3$. Apply Equation~\eqref{eq:5term} with $A=\{s_1,\ldots,s_r\}$, $B=\{t_1\}$, $I=\{t_2\}$ and $J=\{t_3\}$ to obtain the following:
\begin{align*}\Big[\{\{s_1,\ldots,s_r,t_1\},\{t_2,t_3\}\}\Big] &+ \Big[\{\{s_1,\ldots,s_r,t_2\},\{t_1,t_3\}\}\Big] + \Big[\{\{s_1,\ldots,s_r,t_3\},\{t_1,t_2\}\}\Big] \\
&= \Big[\{\{s_1,\ldots,s_r,t_2,t_3\},\{t_1\}\}\Big]+ \Big[\{\{s_1,\ldots,s_r\},\{t_1,t_2,t_3\}\}\Big].
\end{align*}
Then all terms on the left side of the equality become invariants for partitions with a block of size 2 and a block of size $r+1$, and the only non-zero term to the right of the equality is $[\pi]$.

Fix some positive integer $m$. Suppose now that the invariant of any two-block partition with the smaller block of size $k$ and the larger block of size $r$ is in $S^{(2,2,1^{n-4})}$ for $k\leq m$, and we will show the result holds for two-block partitions with the smaller block of size $m+1$ and the second block of size $r$. Let $\pi=\{\{t_1,\ldots,t_{m+1}\}.\{s_1,\ldots,s_r\}\}$ be such a partition. Apply Equation~\eqref{eq:5term} with $A=\{s_1,\ldots,s_r\}$, $B=\{t_1,\ldots,t_{m-1}\}$, $I=\{t_m\}$ and $J=\{t_{m+1}\}$ to obtain the following:
\begin{align*}\Big[\{\{s_1,\ldots,s_r,t_1,\ldots,t_{m-1}\},\{t_m,t_{m+1}\}\}\Big] &+ \Big[\{\{s_1,\ldots,s_r,t_m\},\{t_1,\ldots,t_{m-1},t_{m+1}\}\}\Big]\\ &+ \Big[\{\{s_1,\ldots,s_r,t_{m+1}\},\{t_1,\ldots,t_{m-1},t_m\}\}\Big] \\
= \Big[\{\{s_1,\ldots,s_r,t_m,t_{m+1}\},\{t_1,\ldots,t_{m-1}\}\}\Big]  &+ \Big[\{\{s_1,\ldots,s_r\},\{t_1,\ldots,t_{m-1},t_m,t_{m+1}\}\}\Big].
\end{align*} 
This gives a relation between the invariants of the five following two-block partitions:
\begin{itemize} 
\item a partition with smaller block of size $2$,
\item a partition with a block of size $r+1$ and a block of size $m$, 
\item another partition with a block of size $r+1$ and a block of size $m$,  
\item a partition with a block of size $r+2$ and block of size $m-1$, and
\item our original partition $\pi$.
\end{itemize}
The first invariant is in $S^{(2,2,1^{n-4})}$ by definition, and the second, third, and fourth invariants listed are in $S^{(2,2,1^{n-4})}$ by induction.
This completes the proof for $\pi\in\setpart(n,2)$. The desired result then follows from Corollary~\ref{cor:twoblocks}.
\end{proof}

\begin{lemma}\label{lemma:basis}
The set $\{ [\pi] : \pi \in \nc(n,d) \}$ of invariants of noncrossing partitions of $n$ with $d$ blocks and no singletons is basis of the pennant Specht module $S^{(d,d,1^{n-2d})}$.
\end{lemma}
\begin{proof}
By Lemma~\ref{lemma:invariants_pennant}, each $[\pi]$ is in the pennant Specht module $S^{(d,d,1^{n-2d})}$. By Lemma~\ref{lemma:independent}, they are linearly independent. By \cite[Propositions~2.1 and 2.3]{Pechenik}, there is a bijection between $\nc(n,d)$ and the set of standard Young tableaux of pennant shape $(d,d,1^{n-2d})$, the number of which is the dimension of $S^{(d,d,1^{n-2d})}$. Hence, the $[\pi]$ form a basis.
\end{proof}

For any permutation $w \in S_n$ and any $B \subseteq \{1, \dots, n\}$, define $w \cdot B = \{ w(b) : b \in B \}$.
For any set partition $\pi = \{B_1, \dots, B_d \} \in \setpart(n,d)$ (not necessarily noncrossing), let $w \cdot \pi$ be the set partition with blocks $\{ w \cdot B_1, w \cdot B_2, \dots, w \cdot B_d \}$.

\begin{proposition}
\label{prop:any_perm}
For any set partition $\pi \in \setpart(n,d)$ and any permutation $w \in S_n$, we have 
\[
w \cdot [\pi] = \sgn(w) [w \cdot \pi],
\]
where $\sgn(w)$ denotes the sign of the permutation $w$.
\end{proposition}

\begin{proof}
Let $\pi=\{\pi_1,\ldots,\pi_d\}\in\setpart(n,d)$ and $w\in S_n$. We show the result for $w$ a simple transposition $s_i$. Then by induction, it is true for all $w$. 

Consider $s_i$ acting on $\pi$. There are two cases: Either $i$ and $i+1$ are in the same block of $\pi$ or else they are not. 
 
First, suppose they are not in the same block. The $\pi$ and $s_i \cdot \pi$ differ by exchanging that pair of elements between blocks. For $T\in\rstab(\pi)$, let $\tilde{T}$ denote the tableau obtained from $T$ by applying the permutation $s_i$, swapping the labels $i$ and $i+1$.
Observe that $\{ \tilde{T} : T \in \rstab(\pi) \} = \rstab(s_i \cdot \pi)$.
Note further that $\sgn(T) = -\sgn(\tilde{T})$ for all $T \in \rstab(\pi)$. 
Recall from Definition~\ref{def:polynomial} that \[[\pi] = \sum_{T\in \rstab(\pi)}\sgn(T) \; \rspoly(T).\]
Therefore, 
\[  s_i \cdot [\pi]  =   s_i \cdot \sum_{T\in \rstab(\pi)}\sgn(T) \; \rspoly(T) = \sum_{T\in \rstab(\pi)} \sgn(T) \; \rspoly(\tilde{T}) = \sum_{\tilde{T}\in \rstab(s_i \cdot \pi)} -\sgn(\tilde{T}) \; \rspoly(\tilde{T}) = -[s_i \cdot \pi],\] which yields the desired result in this case.

Now suppose $i$ and $i+1$ are in the same block of $\pi$. Then $\pi = s_i \cdot \pi$, so $[\pi] = [s_i \cdot \pi]$. On the other hand, $s_i \cdot [\pi]$ differs from $[\pi]$ by swapping two adjacent columns of exactly one determinant in each summand. Hence, $s_i \cdot [\pi] = - [\pi] = -[s_i\cdot\pi]$.
\end{proof}

\begin{corollary}
\label{prop:rotation}
Up to signs, the long cycle $c_n = n12\dots (n-1)$ acts by rotation and the long element $w_0$ acts by reflection.

Precisely, for any set partition $\pi \in \setpart_n$, we have 
\[
c_n \cdot [\pi] = (-1)^{n-1} [\mathtt{rot}(\pi)] \quad \text{and} \quad w_0 \cdot [\pi] = (-1)^{\binom{n}{2}} [\mathtt{refl}(\pi)],
\]
where $\mathtt{rot}$ denotes counterclockwise rotation by $(360/n)^\circ$ and $\mathtt{refl}$ denotes reflection across the diameter with endpoint halfway between vertices $n$ and $1$.
\end{corollary}
\begin{proof}
This follows from Proposition~\ref{prop:any_perm} by noting $\sgn(c_n)=(-1)^{n-1}$ and $\sgn(w_0)=(-1)^{\binom{n}{2}}$.
\end{proof}

This completes the proof of Theorem~\ref{thm:main}.

\section{Final remarks}\label{sec:final}

Using our polynomial representatives, one may easily obtain formulas for resolving crossing diagrams as linear combinations of noncrossing ones. In particular, we recover the ``skein'' relations of \cite{Rhoades:skein} (up to signs), as relations $\sigma_1$ through $\sigma_4$ in Figure~\ref{fig:skein_relations}. The relation $\sigma_5$ is included as an example of the more general sort of relations that may easily be extracted from Theorem~\ref{thm:5term}.

\begin{corollary}
Web invariants $[\pi]$ for noncrossing set partitions without singletons obey the ``skein'' relations illustrated in Figure~\ref{fig:skein_relations}.
\end{corollary}
\begin{proof}
Each ``skein'' relation follows from applying the five-term recurrence (Theorem~\ref{thm:5term}) with $A=\{a_1,\ldots,a_k\}$, $B=\{b_1,\ldots, b_m\}$, $I=\{i\}$, and $J=\{i+1\}$ or $\{j\}$. In $\sigma_1$, $\sigma_2$, and $\sigma_3$, we have $k=1$ and/or $m=1$, so some terms of the five-term recurrence will have a singleton block and thus the corresponding invariants will be $0$.
\end{proof}

\begin{figure}[ht]
\includegraphics[width=\textwidth]{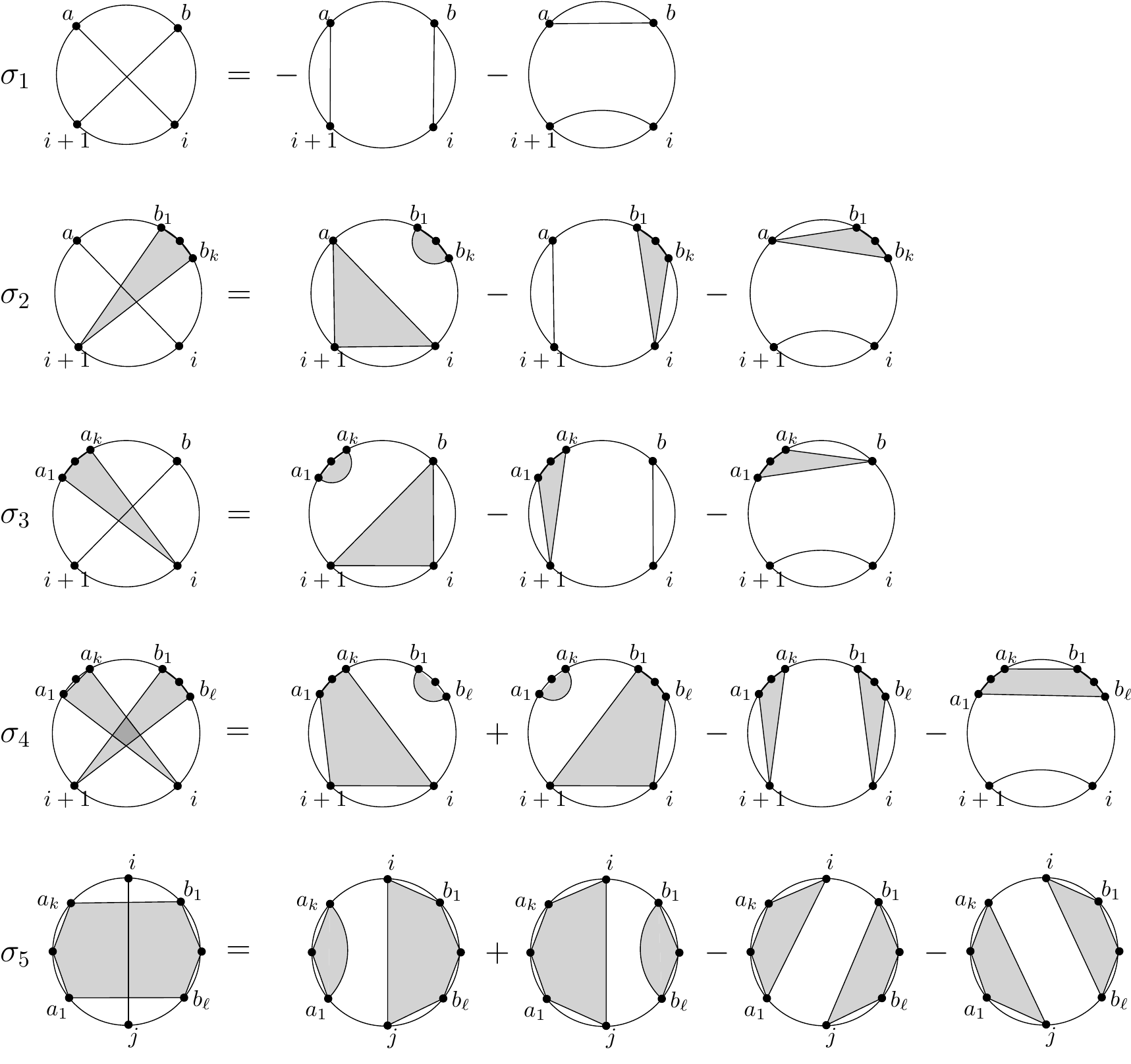}
\caption{A sampling of five relations on invariant polynomials for set partitions, expressing how to write the invariant for a set partition with crossings in terms of the noncrossing basis of Lemma~\ref{lemma:basis}}\label{fig:skein_relations}
\end{figure}

While relations $\sigma_1$ through $\sigma_4$ are sufficient to recursively expand $[\pi]$ for any (crossing) set partition in the noncrossing basis of Lemma~\ref{lemma:basis}, it is inefficient to use this method to expand invariants such as $[\pi]$ for $\pi$ the set partition on the left side of $\sigma_5$. One has to think about the sequence of simple transpositions that would render the diagram noncrossing, and then for each transposition in turn apply one of the rules $\sigma_1$ through $\sigma_4$. In this process, one obtains many terms, most of which cancel. A virtue of our approach is that more general relations such as $\sigma_5$ are easy to obtain.

With this representation theory in hand, it is also not hard to recover the cyclic sievings of \cite[Theorems~1.2 and 1.4]{Pechenik} and \cite[Theorem~7.2]{Reiner.Stanton.White} by similar arguments to \cite[$\mathsection$8]{Rhoades:skein}. The key is that we have identified the combinatorial objects in question (noncrossing partitions) with bases for the Specht modules, such that the long cycle $c_n = n12 \ldots (n-1)$ and the long element $w_0 = n(n-1) \ldots 1$ act (up to signs) by rotation and reflection (Corollary~\ref{prop:rotation}). Therefore, when we restrict attention to the actions of these particular symmetric group elements, the analysis becomes similar to the much easier problem of studying permutation representations. In particular, the characters of the symmetric group elements in the dihedral subgroup $\langle c_n, w_0 \rangle$ on Specht modules for pennant shapes coincide with fixed-point evaluations. One may compute these characters by applying Springer's theorem on regular elements.

The cyclic sievings of \cite[Theorems~1.2 and 1.4]{Pechenik} are explicitly stated in terms of $K$-promotion and $K$-evacuation of $2$-row rectangular packed increasing tableaux. They apply to rotation and reflection of noncrossing set partitions without singletons through the equivariant bijections recalled in Section~\ref{sec:tableaux}. The order of $K$-promotion on $k$-row rectangular tableaux $\inc^q(d^k)$ is generally much larger than $q$ \cite{Pechenik} whenever $k>3$. For $k=2$, we have seen that the order is $q$. In the only remaining case $k=3$, the order is unknown, but conjectured to be exactly $q$ \cite[Conjecture~4.12]{Dilks.Pechenik.Striker}. Resolving this conjecture would also establish that the order of rowmotion on the product $[a] \times [b] \times [3]$ of chain posets is $a+b+2$. We hope that our algebraic analysis of the $k=2$ case will provide additional tools for tackling these conjectures.

\begin{prob}\label{prob:3row}
Develop a bijection from $\inc^q(d,d,d)$ to a family of planar diagrams embedded in a disk with $q$ boundary vertices, which carries $K$-promotion to rotation and $K$-evacuation to reflection.
\end{prob}

A tricky aspect of Problem~\ref{prob:3row} is that the cardinality of $\inc^q(d,d,d)$ does not generally coincide with the dimension of a Specht module (due to the presence of large prime factors). Enumerations of these tableaux were given in \cite{Pressey.Stokke.Visentin,Morales.Pak.Panova:IV}.

To our knowledge, Theorem~\ref{thm:main} is the first development of web bases for a family of Lie groups $P$ of unbounded dimension. We expect this idea to have significant consequences beyond what we have so far explored. We hope that our theory can be extended to describe the algebraic structure of invariants for the quantum group deformation $U_q(\mathfrak{p})$. Assuming this can be accomplished, the diagramatics suggest it might be possible to thereby extract quantum invariants for spatial (hyper)graphs. We similarly hope that our constructions may help describe the topology of the Springer fiber for shape $(d,d,1^\ell)$, extending the work of \cite{Fung,Stroppel.Webster,Russell} for the case $\ell=0$. Web bases in the $\fsl_2$ and $\fsl_3$ cases are important in cluster algebra structure of Grassmannian coordinate rings \cite{Fomin.Pylyavskyy,Fraser}. While cluster algebras from coordinate rings of partial flag varieties (e.g., \cite{Geiss.Leclerc.Schroer}) have been comparatively less studied, we believe that the pieces of these coordinate rings spanned by our polynomials $[\pi]$ for noncrossing set partitions may be a fruitful place to look for extensions of the cluster theory from \cite{Fomin.Pylyavskyy}.

\section*{Acknowledgements}
We thank Julianna Tymoczko for related discussions in the early stages of this project. We're also grateful to Jesse Kim, Brendon Rhoades, and Joshua Swanson for discussions about \cite{Kim.Rhoades}, and to Vic Reiner and Melissa Sherman-Bennett for bringing \cite{Armstrong.Reiner.Rhoades} and \cite{Geiss.Leclerc.Schroer}, respectively, to our attention. OP is grateful for very helpful discussions and correspondence with Jake Levinson and David E Speyer, and to Brendon Rhoades for pointing out an error in our original description of the Lie groups $P$. We thank the anonymous referee for numerous helpful suggestions that streamlined several technical arguments.

OP acknowledges support from NSERC Discovery Grant RGPIN-2021-02391 and Launch Supplement DGECR-2021-00010. JS acknowledges support from Simons Foundation/SFARI grant (527204, JS).

\bibliographystyle{amsalpha}
\bibliography{increasingwebs}

\end{document}